\documentclass[reqno]{amsart}
\usepackage{amssymb,amsmath,amsfonts,amscd,amsthm}
\usepackage{amsbsy,bm}

\newtheorem{theorem}{Theorem}[section]

\newtheorem{definition}{Definiton}[section]

\newtheorem{lemma}{Lemma}[section]
\newtheorem{proposition}{Proposition}[section]
\newtheorem{corollary}{Corollary}[section]
\theoremstyle{remark}
\newtheorem{remark}{Remark}[section]

\begin{document}
\title[Equivariant CR minimal immersions from $S^3$ into $\mathbb CP^n$]
{Equivariant CR minimal immersions from $S^3$\\ into $\mathbb CP^n$}

\author{Zejun Hu, Jiabin Yin}
\address{%
School of Mathematics and Statistics, Zhengzhou University, Zhengzhou 450001, People's Republic of China}
\email{huzj@zzu.edu.cn; welcomeyjb@163.com}
\author{Zhenqi Li}
\address{%
Department of Mathematics, Nanchang University, Nanchang 330047, People's Republic of China}
\email{zhenqili@263.net}

\thanks{The first two authors were supported by grants of NSFC-11371330 \& 11771404, and the third author
was supported by grants of NSFC-11361041.}

\date{}


\begin{abstract}
The equivariant CR minimal immersions from the round $3$-sphere
$S^3$ into the complex projective space $\mathbb CP^n$ have been
classified by the third author explicitly (J London Math Soc 68:
223-240, 2003). In this paper, by employing the equivariant
condition which implies that the induced metric is left-invariant,
and that all geometric properties of $S^3={\rm SU}(2)$ endowed with
a left-invariant metric can be expressed in terms of the structure
constants of the Lie algebra $\mathfrak{su}(2)$, we establish an
extended classification theorem for equivariant CR minimal
immersions from the $3$-sphere $S^3$ into $\mathbb CP^n$ without the
assumption of constant sectional curvatures.
\end{abstract}

\maketitle
\numberwithin{equation}{section}
\section{Introduction}\label{sect:1}
\abovedisplayskip3pt plus2pt minus2pt \belowdisplayskip3pt plus2pt minus2pt\parskip1pt

Let $\mathbb CP^n$ denote the $n$-dimensional complex projective space with
almost complex structure $J$ and constant holomorphic sectional
curvature $4$. It is known that there are considerable researches
for the submanifolds of $\mathbb CP^n$, e.g. we would mention the study of
its holomorphic, or (parallel) Lagrangian submanifolds in
\cite{C-D-V-V-1,DLVW,Na,N1,N2,O}, among a great many others. In
particular, the interesting and important article of
Bolton-Jensen-Rigoli-Woodward \cite{B-J-R-W}, which, by drawing upon
Wolfson's notion of a harmonic sequence and also exhibiting a number
of prominent features of the related foundational work of Boruvka,
Calabi, Chern, Eells, Wood and Lawson (see references as cited
therein), had deeply investigated the conformal minimal immersions
of the $2$-sphere ${S}^2$ into $\mathbb CP^n$ with constant
curvature and constant K\"ahler angle. As a matter of fact, the
paper \cite{B-J-R-W} had brought enormous impact for subsequent
research.

To extend the research of \cite{B-J-R-W}, and as a counterpart of
Mashimo \cite{Ma} studying the immersions from $S^3$ into
$S^n$, a variety of immersions of $S^3$ into $\mathbb CP^n$ have been
considered. For related references we refer to
\cite{C-D-V-V,F-P,H-Li,Li,L-H,L-P,L-T,L-V-W}, among which the paper
of the third author \cite{Li} is of fundamental importance for us.

Following \cite{Be}, an immersed submanifold $\varphi:{M}\rightarrow{\mathbb CP^n}$
is called a {\it CR-submanifold} if $TM$ can be decomposed into an
orthogonal direct sum $TM=V_1\oplus V_2$ such that $JV_1\subset
T^\perp M$ and $JV_2=V_2$. In this case the immersion $\varphi$ is said
to be {\it CR type}.

Recall that $S^3=\{(z,w)\in\mathbb C^2|z\bar{z}+w\bar{w}=1\}$ can be
identified with the Lie group $SU(2)$ in a standard way. By
definition in \cite{Li}, a mapping $\varphi:{S^3=SU(2)}\rightarrow\mathbb CP^n$ is
said to be {\it equivariant} if $\varphi$ is compatible with the
structure of Lie group $SU(2)$. For details see Sect.\ \ref{sect:3}.

In \cite{Li}, among many others, the third author studied
equivariant CR minimal $S^3$ immersed in $\mathbb CP^n$ with constant
sectional curvature $c$ (Theorem 6.1 in \cite{Li}). It was proved
that if such an immersion $\varphi:S^3\rightarrow\mathbb CP^n$ is linearly full,
then $n=2m^2-3$ for some integer $m\geq 2$ and $c=1/(m^2-1)$.
Moreover, up to an isometry of $S^3$ and a holomorphic isometry
of $\mathbb CP^n$, the immersion $\varphi$ is exactly the immersion defined in
Example 4.4 of \cite{Li}. In \cite{L-H}, Li and Huang made an
interesting advance by showing that any minimal CR $3$-sphere in
$\mathbb CP^n$ is actually equivariant if the induced metric of $S^3$ has
constant sectional curvature.

In this paper, following closely \cite{Li}, we continue with the
study of equivariant CR minimal immersion from $S^3$ into $\mathbb CP^n$
but without the assumption that the induced metric on $S^3$ has
constant sectional curvatures. By employing the equivariant
condition, Proposition 4.2 of \cite{Li} which implies that the
induced metric is left-invariant, and that all geometric properties
of $S^3={\rm SU}(2)$ endowed with a left-invariant metric can be
expressed in terms of the structure constants of the Lie algebra
$\mathfrak{su}(2)$, we can finally overcome the difficulty brought
by the missing of constant sectional curvatures condition to achieve
the goal of a new classification theorem. Our main results can be
summarized as follows:
\begin{theorem}\label{thm:1.1}
Let $\varphi: {S^3}\rightarrow{\mathbb C}P^{n}\ (n\ge2)$ be
a linearly full equivariant CR minimal immersion with induced
metric $ds^2$. Then, up to an inner
automorphism of ${\rm SU}(2)$ and a holomorphic isometry of
$\mathbb{C}P^n$, the immersion $\varphi$ can be expressed as
as one of the following three immersions:
\begin{enumerate}
\item[(1)]
$n=2$, and $(S^3,ds^2)$ is not a Berger sphere with
$$
\varphi(z,w)=[\cos\tfrac{\pi}8(z^2,\sqrt{2}zw,w^2)
+\sqrt{-1}\sin\tfrac{\pi}8(\bar{w}^2,-\sqrt{2}\bar{z}\bar{w},\bar{z}^2)].
$$
\item[(2)]
$n\ge2$, and $(S^3,ds^2)$ is a Berger sphere with
$$
\varphi(z,w)=\big[\cos t \sum_{\alpha=0}^k\sqrt{C_{\alpha}^k}z^{k-\alpha}w^\alpha\varepsilon_\alpha+
\sqrt{-1}\sin t\sum_{\alpha'=0}^{\ell}\sqrt{C_{\alpha'}^\ell}z^{\ell-\alpha'}w^{\alpha'}\varepsilon'_{\alpha'}\big],
$$
where $t\in (0,\pi/2)$ is determined by, for nonnegative integer
$k,\ell$ satisfying $k-\ell>0$ and $k+\ell+1=n$, $ \tan^2
t=2k/\big[3(k-\ell)+\sqrt{(k+\ell)^2+8(k-\ell)^2}\big],$
and
$\{\varepsilon_0,\ldots,\varepsilon_k,\varepsilon_0',\ldots,\varepsilon_{\ell}'\}$
is the natural basis of $\mathbb{C}^{n+1}=\mathbb
C^{k+1}\oplus\mathbb C^{\ell+1}$. In particular, if
$\tan^2t=k/(2k-\ell+4)$, $(S^3,ds^2)$ has constant sectional
curvature $c=1/(m^2-1)$ for some integer $m=(k-\ell)/2\geq2$.
\end{enumerate}
\end{theorem}
This paper is organized as follows. In Sect. \ref{sect:2}, we
consider the left-invariant metric $ds^2$ on $S^3=SU(2)$ with Lie
algebra $\mathfrak{su}(2)$ and recall that the connection and
curvature of $ds^2$ can be expressed precisely by the structure
constants of $\mathfrak{su}(2)$. In particular, in terms of the
structure constants, we state the criterion of $ds^2$ being a Berger
metric (Proposition \ref{prop:2.2}) and the criterion for a
left-invariant metric $ds^2$ on $S^3$ to have constant sectional
curvatures $c$ (Proposition \ref{prop:2.3}). In Sect.\ \ref{sect:3},
we introduce basic formulae for equivariant CR immersion (Lemma
\ref{lem:3.1}) and derive the classical Gauss-Weingarten formulae
and Gauss-Codazzi equations. In Sect.\ \ref{sect:4}, a roughly
classification for all equivariant CR minimal immersions is proved.
It turns out that there are only three possibilities of such
immersions: a non-Berger sphere and $n=2$; a Berger sphere with
$\nabla^\bot\xi_0=0$ for each $n\ge2$; Berger spheres with
$\nabla^\bot\xi_0\ne0$ for each $n\ge3$ (see Theorem \ref{thm:4.1}).
Here $\xi_0=J(\varphi_{\ast}(X_1))$ is a normal vector and $X_1$ is
an unitary vector field that spans $W_1$. Then in Sect.\
\ref{sect:5}, it is shown that all these possibilities do occur and
their explicitly expressions in polynomial are presented
(Propositions \ref{prop:5.1} and \ref{prop:5.2}). Finally in Sect.\
\ref{sect:6}, we demonstrate the uniqueness for such immersions
(Theorems \ref{thm:6.1} and \ref{thm:6.2}).

\vskip 2mm

\noindent{\bf Acknowledgements}. The authors are greatly indebted to
the referee for his/her very helpful suggestions and valuable
comments.

\section{The Lie group $S^3$ and its canonical
structures}\label{sect:2}

In this section, we shall review some fundamental results of
$S^3$ which are given in \cite{Li} and will be needed in later
sections. For completeness, we also give several new results which
are of independent meaning.

\subsection{The standard metric of $S^3$}\label{sect:2.1}~

Let $S^3=\{(z,w)\in\mathbb C^2|z\bar{z}+w\bar{w}=1\}$. It is
identified with $SU(2)$ by
\begin{equation}\label{eqn:2.1}
{\rm id}:{S^3}\rightarrow {SU(2)}:{(z,w)}\mapsto{
\begin{pmatrix}
z & -\bar{w}\\  w & \ \bar{z}
\end{pmatrix}}.
\end{equation}

Denote by $\mathfrak{su}(2)$ and $\mathfrak{su}(2)^*$ the Lie
algebra consisting of all left-invariant vector fields and the
vector space of all left-invariant $1$-forms on $S^3$,
respectively. The collection of left-invariant $1$-forms
$\{\omega'_1,\omega'_2,\omega'_3\}$, defined by
\begin{equation}\label{eqn:2.2}
\begin{pmatrix}
i\omega'_1 & -\omega'_2+i\omega'_3\\  \omega'_2+i\omega'_3 &-i\omega'_1
\end{pmatrix}
=
\begin{pmatrix}
\bar z & \bar w\\  -w & z
\end{pmatrix}
\begin{pmatrix}
dz & -d\bar w\\ dw & d\bar z
\end{pmatrix},
\end{equation}
forms a global frame of $T^*S^3$. Here, and throughout this
paper, we adopt the usual notation $i=\sqrt{-1}$. Exterior
differentiation of \eqref{eqn:2.2} leads to the equations
\begin{equation}\label{eqn:2.3}
d\omega'_1=2\,\omega'_2\wedge \omega'_3,\ \ d\omega'_2=2\,\omega'_3 \wedge \omega'_1,\
\ d\omega'_3=2\,\omega'_1\wedge \omega'_2.
\end{equation}

From \eqref{eqn:2.2} we get
\begin{equation}\label{eqn:2.4}
\left\{
\begin{aligned}
\omega'_1&=-\tfrac{i}2(\bar{z}dz+\bar{w}dw-zd\bar{z}-wd\bar{w}),\\
\omega'_2&=-\tfrac 12(wdz-zdw+\bar{w}d\bar{z}-\bar{z}d\bar{w}),\\
\omega'_3&=\tfrac{i}2(wdz-zdw-\bar{w}d\bar{z}+\bar{z}d\bar{w}).
\end{aligned}
\right.
\end{equation}
Thus the dual frame $\{X'_1,X'_2,X'_3\}$ of $\{\omega'_1,\omega'_2,\omega'_3\}$ is given by
\begin{equation}\label{eqn:2.5}
\begin{aligned}
X'_1(z,w)=i(z,w),\ X'_2(z,w)=(-\bar{w},\bar{z}),\
X'_3(z,w)=i(-\bar{w},\bar{z}).
\end{aligned}
\end{equation}

The standard metric on $S^3$ is defined by
$ds_0^2=\sum_i\omega'_i\otimes\omega'_i$, which is bi-invariant and has
constant sectional curvature $1$. Choose an orientation of $S^3$
such that $\omega'_1\wedge\omega'_2\wedge\omega'_3=*1$, where $*1$ is the
volume element with respect to $ds_0^2$. An orthonormal frame
$\{\tilde\omega'_1,\tilde\omega'_2,\tilde\omega'_3\}$ of $T^*S^3$ is called {\it
oriented frame} if
$\tilde{\omega}'_1\wedge\tilde{\omega}'_2\wedge\tilde{\omega}'_3=*1$. If otherwise,
$\tilde{\omega}'_1\wedge\tilde{\omega}'_2\wedge\tilde{\omega}'_3=-*1$, it is called
the {\it oppositive oriented frame}.

With respect to $ds_0^2$, the Hodge star operator
$*:{\mathfrak{su}(2)^*}\rightarrow{\mathfrak{su}(2)^*\wedge
\mathfrak{su}(2)^*}$ is a linear isomorphism, which maps the basis
$\{\omega'_1,\omega'_2,\omega'_3\}$ of $\mathfrak{su}(2)^*$ into a basis
$\{\omega'_2\wedge\omega'_3,\omega'_3\wedge\omega'_1, \omega'_1\wedge\omega'_2\}$ of
$\mathfrak{su}(2)^*\wedge \mathfrak{su}(2)^*$. Thus, by
\eqref{eqn:2.3}, $d=2*$ and it is also a linear isomorphism.

The following lemma can be proved easily.
\begin{lemma}\label{lem:2.1}
Suppose $\{\tilde\omega'_1,\tilde\omega'_2,\tilde\omega'_3\}$ is a basis of
$\mathfrak{su}(2)^*$. Then it is an oriented frame of $T^*S^3$
{\rm (}related to $ds_0^2${\rm)} if and only if it satisfies
\eqref{eqn:2.3}.
\end{lemma}

The adjoint representation ${\rm Ad}$ of $SU(2)$ induces an action
of the Lie group $SU(2)$ on $\mathfrak{su}(2)^*$ by
$$
{\rm Ad}^*:{SU(2)\times
\mathfrak{su}(2)^*}\rightarrow{\mathfrak{su}(2)^*}:{(\varrho,\theta)}\mapsto{{\rm
Ad}_{\varrho}^*(\theta)=a_{\varrho}^*\theta},
$$
where $a_\varrho:SU(2)\rightarrow SU(2)$ defined by
$a_\varrho(A)=\varrho^{-1}A\varrho$ is an inner automorphism. Since
$d\circ a_\varrho^*=a_\varrho^*\circ d$, by virtue of Lemma
\ref{lem:2.1}, it is easily seen that $a_\varrho^*$ maps an oriented
frame $\{\tilde\omega'_1,\tilde\omega'_2, \tilde\omega'_3\}$ into an
oriented frame
$\{a_\varrho^*\tilde\omega'_1,a_\varrho^*\tilde\omega'_2,
a_\varrho^*\tilde\omega'_3\}$. Conversely, one can prove the
following lemma.
\begin{lemma}\label{lem:2.2}
For any given oriented frame
$\{\tilde\omega'_1,\tilde\omega'_2,\tilde\omega'_3\}$ of $T^*S^3$,
there exists an element $\varrho\in SU(2)$ such that
$a_\varrho^*\tilde\omega'_1=\omega'_1$,
$a_\varrho^*\tilde\omega'_2=\omega'_2$ and
$a_\varrho^*\tilde\omega'_3=\omega'_3$.
\end{lemma}

\begin{remark}\label{rem:2.1}
{\rm Lemma \ref{lem:2.2} implies that $SU(2)$ acts transitively on
$\mathfrak{su}(2)^*$ by $Ad^*$, and it also acts transitively on the
set of all orientated frame.}
\end{remark}

\subsection{The structure constants of the Lie algebra
$\mathfrak{su}(2)$}\label{sect:2.2}~

Now, we examine the properties of a general left-invariant metric on
$S^3$. Some results in this subsection expand that of Milnor
\cite{Mil}.

From now on, we adopt the convention of range of indices:
$i,j,k,\ldots=1,2,3$.

A left-invariant metric on $S^3=SU(2)$ has the form that
$ds^2=\sum \omega_i\otimes\omega_i$, where $\{\omega_i\}$ is a basis of
$\mathfrak{su}(2)^*$. Obviously, $\{\omega_i\}$ is a global orthonormal
frame of $T^*S^3$.

Related to $\{\omega_i\}$, the structure equations of $ds^2$ are given
by
\begin{equation}\label{eqn:2.6}
d\omega_i=-\sum\omega_{ij}\wedge\omega_j,\ \ \omega_{ij}+\omega_{ji}=0,
\end{equation}
\begin{equation}\label{eqn:2.7}
d\omega_{ij}=-\sum\omega_{ik}\wedge\omega_{kj}+\Omega_{ij},
\end{equation}
where $\omega_{ij}$ and $\Omega_{ij}$ are the connection forms and
curvature forms, respectively.

Let $\{X_i\}$ be the dual frame of $\{\omega_i\}$ defined on $S^3$.
Denote by $\nabla$ the Levi-Civita connection of $ds^2$. Then we
have
\begin{equation}\label{eqn:2.8}
\nabla_{X_k} X_i=-\sum\omega_{ij}(X_k) X_j,\ \
\nabla_{X_k}\omega_i=-\sum\omega_{ij}(X_k)\omega_j.
\end{equation}

Let $\{c^k_{ij}\}$ be the structure constants of $\mathfrak{su}(2)$
with respect to the basis $\{X_1, X_2, X_3\}$, i.e., $[X_i,
X_j]=\sum c^k_{ij} X_k$. It is clear that $c^k_{ij}=-c^k_{ji}$.

Using the Koszul formula one gets
\begin{equation}\label{eqn:2.9}
\omega_{ij}=\tfrac12\sum(c_{ij}^k-c_{jk}^i+c_{ik}^j)\omega_k.
\end{equation}

Equivalently, the above relations can be written exactly as
\begin{equation}\label{eqn:2.10}
\left\{
\begin{aligned}
\omega_{12}&=c_{12}^1\omega_1+c_{12}^2\omega_2+(c_{12}^3+a)\omega_3,\\
\omega_{23}&=(c_{23}^1+a)\omega_1+c_{23}^2\omega_2+c_{23}^3\omega_3,\\
\omega_{31}&=c_{31}^1\omega_1+(c_{31}^2+a)\omega_2+c_{31}^3\omega_3,
\end{aligned}
\right.
\end{equation}
where
\begin{equation}\label{eqn:2.11}
a=-\tfrac12(c_{23}^1+c_{31}^2+c_{12}^3).
\end{equation}

Hence, $\omega_{ij}\in \mathfrak{su}(2)^*$, and therefore, by
\eqref{eqn:2.7}, $\Omega_{ij}\in \mathfrak{su}(2)^*\wedge
\mathfrak{su}(2)^*$ and it is also left-invariant. Substituting
\eqref{eqn:2.9} into \eqref{eqn:2.6}, we can write
\begin{equation}\label{eqn:2.12}
d\omega_i=-\tfrac12\sum
c_{jk}^i\omega_j\wedge\omega_k=-\sum_{j<k}c_{jk}^i\omega_j\wedge\omega_k,
\end{equation}
i.e.,
\begin{equation}\label{eqn:2.13}
\left\{
\begin{aligned}
-d\omega_1&=c^1_{23}\omega_2\wedge\omega_3+c^1_{31}\omega_3\wedge\omega_1+c^1_{12}\omega_1\wedge\omega_2,\\
-d\omega_2&=c^2_{23}\omega_2\wedge\omega_3+c^2_{31}\omega_3\wedge\omega_1+c^2_{12}\omega_1\wedge\omega_2,\\
-d\omega_3&=c^3_{23}\omega_2\wedge\omega_3+c^3_{31}\omega_3\wedge\omega_1+c^3_{12}\omega_1\wedge\omega_2.
\end{aligned}
\right.
\end{equation}

Since $d:{\mathfrak{su}(2)^*}\rightarrow{\mathfrak{su}(2)^*\wedge
\mathfrak{su}(2)^*}$ is an isomorphism, the coefficients matrix on
the right hand of \eqref{eqn:2.13}:
\begin{equation}\label{eqn:2.14}
C=\begin{pmatrix}  c^1_{23}& c^1_{31} &c^1_{12}\\  c^2_{23}& c^2_{31} &c^2_{12}\\  c^3_{23}& c^3_{31} &c^3_{12}
\end{pmatrix}
\end{equation}
is non-degenerate, namely $\det C\ne0$.

It is obvious that the matrix $C$ depends on the frame $\{\omega_i\}$.
Assume $\{\tilde\omega_i\}$ is another orthonormal basis of
$\mathfrak{su}(2)^*$, then a similar relations as \eqref{eqn:2.13}
holds with corresponding matrix $\widetilde C$. Now, we derive the
relationship between $C$ and $\widetilde C$.

Take an orientation for $S^3$. The volume element of $ds^2$ is
still denoted by $*1$. With respect to $ds^2$, the Hodge star
operator ${*}:{\mathfrak{su}(2)^*}\rightarrow{\mathfrak{su}(2)^*\wedge
\mathfrak{su}(2)^*}$ is also a linear isomorphism. Let $\omega_1\wedge
\omega_2\wedge \omega_3=\varepsilon
*1$, $\varepsilon=\pm 1$. Then, by \eqref{eqn:2.13},
$$
(d\omega_1,d\omega_2,d\omega_3)=-\varepsilon (*\omega_1,*\omega_2,*\omega_3)C.
$$

Similarly, if $\{\tilde\omega_i\}$ is another orthonormal frame with
$\tilde\omega_1\wedge\tilde\omega_2\wedge\tilde\omega_3=\tilde\varepsilon *1$, then
$$
(d\tilde\omega_1,d\tilde\omega_2,d\tilde\omega_3)=-\tilde\varepsilon (*\tilde\omega_1,*\tilde\omega_2,*\tilde\omega_3)\widetilde C.
$$

Noting that there exists $T\in O(3)$ such that
$(\tilde\omega_1,\tilde\omega_2,\tilde\omega_3)=(\omega_1,\omega_2,\omega_3)T$, then by the
linearity of $d$ and $*$ we have
$$
-\tilde\varepsilon (*\omega_1,*\omega_2,*\omega_3)T\widetilde C=-\tilde\varepsilon
(*\tilde\omega_1,*\tilde\omega_2,*\tilde\omega_3)\widetilde C=-\varepsilon (*\omega_1,*\omega_2,*\omega_3)CT.
$$

It follows that
\begin{equation}\label{eqn:2.15}
\widetilde
C=\tilde\varepsilon\varepsilon\,T^{-1}CT=\tilde\varepsilon\varepsilon\,T^tCT.
\end{equation}

We will call an orthonormal frame $\{\omega_i\}$ {\it a normalized
frame} if it corresponds to a diagnalized matrix $C={\rm
diag}\,(c_{23}^1,c_{31}^2,c_{12}^3)$.

Before ending this subsection, we summarize the useful results in
the following proposition.

\begin{proposition}\label{prop:2.1}
{\rm (i)} The matrix $C$ in \eqref{eqn:2.14} is symmetric. Consequently,
\begin{equation}\label{eqn:2.16}
c_{21}^2+c_{31}^3=0,\ \ c_{12}^1+c_{32}^3=0,\ \ c_{13}^1+c_{23}^2=0.
\end{equation}

{\rm (ii)} There is always a normalized frame $\{\omega_i\}$ for $ds^2$.

{\rm (iii)} The scalar $a=-\tfrac12{\rm trace}\,C$ defined by
\eqref{eqn:2.11} is an orientation invariant. Moreover, $a^2$ is an
intrinsic invariant of $ds^2$.

{\rm (iv)} If $\{\omega_i\}$ is a normalized frame, then
$c_{jk}^ic_{ij}^k>0$ for distinct $i,j,k$.
\end{proposition}

\begin{proof}
{\rm (i)}. Since $\omega_1\wedge\omega_2\in
\mathfrak{su}(2)^*\wedge \mathfrak{su}(2)^*$ and that $
d:{\mathfrak{su}(2)^*}\rightarrow{\mathfrak{su}(2)^*\wedge \mathfrak{su}(2)^*}$
is an isomorphism, we have $d(\omega_1\wedge\omega_2)=0$ and
$d\omega_1\wedge\omega_2=\omega_1\wedge d\omega_2$. Then, by \eqref{eqn:2.13},
$$
-c_{31}^1\omega_1\wedge\omega_2\wedge\omega_3=d\omega_1\wedge\omega_2=\omega_1\wedge d\omega_2=-c_{23}^2\omega_1\wedge\omega_2\wedge\omega_3.
$$
Thus, $c_{31}^1=c_{23}^2$. Similarly, we have $c_{12}^1=c_{23}^3$
and $c_{12}^2=c_{31}^3$. Hence \eqref{eqn:2.16} follows.

From \eqref{eqn:2.15} and (i), and according to the theory of linear
algebra, we obtain immediately {\rm (ii)} and {\rm (iii)}.

{\rm (iv)}. Let $\{\omega'_i\}$ be an oriented frame of the standard
metric $ds_0^2$. Since both $\{\omega_i\}$ and $\{\omega'_i\}$ are basis
of $\mathfrak{su}(2)^*$, there is a matrix $A=(a_{ij})\in
GL(3,\mathbb{R})$ such that $(\omega_1,\omega_2,\omega_3)=(\omega'_1,\omega'_2,\omega'_3)A$.
Then $\omega_1\wedge\omega_2\wedge\omega_3=(\det
A)\,\omega'_1\wedge\omega'_2\wedge\omega'_3$ and, by Lemma \ref{lem:2.1}, it
holds that
$$
d\omega_1=2a_{11}\omega'_2\wedge\omega'_3+2a_{21}\omega'_3\wedge\omega'_1+2a_{31}\omega'_1\wedge\omega'_2.
$$

On the other hand, by assumption,
$d\omega_1=-c_{23}^1\omega_2\wedge\omega_3$, thus
$$
-c_{23}^1\omega_1\wedge\omega_2\wedge\omega_3=\omega_1\wedge d\omega_1=2(a_{11}^2+a_{21}^2+a_{31}^2)\omega'_1\wedge\omega'_2\wedge\omega'_3.
$$
This implies that $c_{23}^1\det A<0$. Similarly, we get
$c_{31}^2\det A<0$ and $c_{12}^3\det A<0$. Then the assertion
immediately follows.
\end{proof}

\subsection{The Berger metric on $S^3$}\label{sect:2.3}~

First we state the following definition of Berger sphere (cf.
\cite{HLW,T})

\begin{definition}\label{def:2.1}
Let $\{\omega'_i\}$ be an oriented frame of $S^3$ with respect to
$ds^2_0$, and $b,c$ are two positive real numbers. Then, the
left-invariant metric defined by
\begin{equation}\label{eqn:2.17}
ds^2=\tfrac1c(b^2\omega'^2_1+\omega'^2_2+\omega'^2_3)
\end{equation}
is called a {\it Berger metric}.  The $3$-sphere $S^3$ equipped with a Berger metric is called a {\it Berger sphere}.
\end{definition}

\begin{remark}\label{rem:2.2}
{\rm In \cite{HLW,T}, a Berger metric is defined essentially in the
form \eqref{eqn:2.17} with the frame $\{\omega'_i\}$ being defined
by \eqref{eqn:2.2}. Thus, according to Lemma \ref{lem:2.2}, up to an
inner automorphism ${a_\varrho}:{SU(2)}\rightarrow{SU(2)}$, the
above definition of Berger sphere should be the same as that in
\cite{HLW,T}. Note that $b$ and $c$ are two parameters in a Berger
metric $ds^2$, and its sectional curvatures are constant if and only
if $b=1$.}
\end{remark}

Take $\omega'_1\wedge\omega'_2\wedge\omega'_3$ as the orientation of
$S^3$. Then, for the metric \eqref{eqn:2.17},
$$
\{ \omega_1=\tfrac b{\sqrt c}\omega'_1,\ \omega_2=\tfrac1{\sqrt c}\omega'_2,\
\omega_3=\tfrac1{\sqrt c}\omega'_3 \}
$$
is an orthonormal oriented frame. Using Lemma \ref{lem:2.1} we get
\begin{equation}\label{eqn:2.18}
d\omega_1=2b\sqrt{c}\omega_2\wedge\omega_3,\ \
d\omega_2=\tfrac2b\sqrt{c}\omega_3\wedge\omega_1,\ \
d\omega_3=\tfrac2b\sqrt{c}\omega_1\wedge\omega_2.
\end{equation}

Conversely, we have the following criterion for Berger sphere.
\begin{proposition}\label{prop:2.2}
A left-invariant metric $ds^2$ on $S^3$ is a Berger metric if and
only if there exists an orthonormal frame $\{\omega_i\}$ and
constants $c_1,c_2\in\mathbb{R}$ such that
\begin{equation}\label{eqn:2.19}
d\omega_1=2c_1\omega_2\wedge\omega_3,\ \ d\omega_2=2c_2\omega_3\wedge\omega_1,\ \
d\omega_3=2c_2\omega_1\wedge\omega_2.
\end{equation}
\end{proposition}
\begin{proof}
The ``only if" part is trivially from \eqref{eqn:2.18}. To prove the
``if" part, we first note by replacing $\omega_1$ with $-\omega_1$
if necessary we may assume $c_1>0$. Then the item {\rm (iv)} of
Proposition \ref{prop:2.1} implies that $c_2>0$. Let $c=c_1c_2$,
$b=\sqrt{c_1/c_2}$. Then $c_1=b\sqrt{c}$, $c_2=\sqrt{c}/b$, and
\eqref{eqn:2.19} becomes \eqref{eqn:2.18}. Put $\omega'_1=(\sqrt
c/b)\omega_1$, $\omega'_2=\sqrt c\omega_2$, $\omega'_3=\sqrt
c\omega_3$, then $\{\omega'_i\}$ is a basis of $\mathfrak{su}(2)^*$
which satisfies \eqref{eqn:2.3}. According to Lemma \ref{lem:2.1},
$\{\omega'_i\}$ is an oriented frame with respect to $ds_0^2$. Thus
the metric $ds^2=\sum\omega_i^2$ assumes the form \eqref{eqn:2.17}.
\end{proof}

\subsection{The curvature of left-invariant metrics on
$S^3$}\label{sect:2.4}~

As continuation of the previous subsections, we now compute the
curvatures of the left-invariant metric $ds^2$. As the main result,
we give the necessary and sufficient conditions under which $ds^2$
is of constant sectional curvatures.

From \eqref{eqn:2.7}, \eqref{eqn:2.9} and \eqref{eqn:2.12} we have
$$
\Omega_{ij}=-\tfrac14\sum [(c_{ik}^l+c_{lk}^i+c_{il}^k)(c_{jk}^m+c_{jm}^k+c_{mk}^j)
      +(c_{ij}^k+c_{kj}^i+c_{ik}^j)c_{lm}^k]\omega_l\wedge\omega_m.
$$

Let $R_{ijlm}$ denote the components of the Riemannian curvature
tensor of $ds^2$, thus $\Omega_{ij}=\tfrac12\sum
R_{ijlm}\omega_l\wedge\omega_m$ and $R_{ijlm}=-R_{ijml}$. We use
notation $i,j,m\neq$ to indicate $i,j,m$ are distinct indices. Since
$i,j,l,m\in\{1,2,3\}$, if $R_{ijlm}\ne0$ then $l$ has to be $i$ or
$j$, and $i,j,m\neq$. Making use \eqref{eqn:2.11} and
\eqref{eqn:2.16}, we obtain
\begin{equation}\label{eqn:2.20}
\left\{
\begin{aligned}
&R_{ijim}=c_{ij}^j(c_{ij}^m-c_{jm}^i+c_{mi}^j)-2c_{ij}^ic_{im}^i,\\
&R_{ijij}=a^2+(c_{im}^i)^2-(c_{ij}^i)^2-(c_{ij}^j)^2-(c_{ij}^m)^2-c_{jm}^ic_{mi}^j,
\end{aligned}
\right. \ \  i,j,m\neq.
\end{equation}

In particular, if $\{\omega_i\}$ is a normalized frame, then
\begin{equation}\label{eqn:2.21}
R_{ijim}=0,\ \ R_{ijij}=a^2-(c_{ij}^m)^2-c_{jm}^ic_{mi}^j, \ \
i,j,m\neq.
\end{equation}

The following result, which is of independent meaning, establishes a
criterion in terms of the structure constants for the left-invariant
metric $ds^2$ to have constant sectional curvatures $c$.

\begin{proposition}\label{prop:2.3}
A left-invariant metric $ds^2$ has constant curvature $c$ if and
only if for any orthonormal frame $\{\omega_i\}$, the corresponding
structure constants $\{c_{ij}^k\}$ satisfy
\begin{equation}\label{eqn:2.22}
c_{ij}^i=0,\ \ c_{jm}^i=c_{mi}^j=c_{ij}^m=\pm2\sqrt{c}, \ \
i,j,m\neq.
\end{equation}
\end{proposition}

\begin{proof} First of all, if a metric on $S^3$ has
constant sectional curvature $c$, then by Cartan-Hadamard theorem we
have $c>0$.

The ``if" part is a direct consequence of \eqref{eqn:2.20},
\eqref{eqn:2.22} and \eqref{eqn:2.11}.

To prove the ``only if" part, by virtue of Proposition
\ref{prop:2.1}, we can first choose a normalized frame
$\{\tilde\omega_i\}$ and denote by $\{\tilde c^k_{ij}\}$ the
corresponding structure constants. Then by \eqref{eqn:2.21} we have
\begin{equation}\label{eqn:2.23}
\textstyle{
(\tilde{c}_{ij}^m)^2+\tilde{c}_{mi}^j\tilde{c}_{jm}^i=(\tilde{c}_{mi}^j)^2+\tilde{c}_{jm}^i\tilde{c}_{ij}^m=
(\tilde{c}_{jm}^i)^2+\tilde{c}_{ij}^m\tilde{c}_{mi}^j=a^2-c,\ \ i,j,m\neq}.
\end{equation}

If $\tilde{c}_{mi}^j,\tilde{c}_{ij}^m,\tilde{c}_{jm}^i$ are mutually
distinct, then by \eqref{eqn:2.23}
$\tilde{c}_{mi}^j=\tilde{c}_{ij}^m=\tilde{c}_{jm}^i=0$, a
contradiction. So, without loss of generality, we may assume that
$\tilde c_{mi}^j=\tilde c_{ij}^m$. By \eqref{eqn:2.23} again we have
$\tilde c_{mi}^j=\tilde c_{ij}^m=\tilde c_{jm}^i=\pm2\sqrt{c}$.
Thus, $\widetilde{C}=\pm2\sqrt{c}I$ with $I$ denoting the unit
matrix. Finally, let $\{\omega_i\}$ be an arbitrary orthonormal
frame, then by \eqref{eqn:2.15} we have $C=\varepsilon
T\widetilde{C}T^{-1}=\pm2\varepsilon\sqrt{c}I$ with
$\varepsilon=\pm1$. The assertion is verified.
\end{proof}

\begin{remark}\label{rem:2.3}
Proposition \ref{prop:2.3} implies that a left-invariant metric on
 $S^3$ is homothetic to the round one if and only if $C$ is proportional to
the identity. Similarly, as indicated by Proposition \ref{prop:2.2},
a left-invariant nonconstant sectional curvature metric on $S^3$ is
a Berger metric if and only if the matrix $C$ has exactly two
nonzero distinct eigenvalues.
\end{remark}

\section{Equivariant CR immersions from $S^3$ into $\mathbb CP^n$}\label{sect:3}

In general, an {\it equivariant} map $\varphi:{G_1/H_1}\rightarrow{G_2/H_2}$ is
such a map which is compatible with the Lie group structures of
$G_1$ and $G_2$ (cf. \cite{Li}). In particular, for
$S^3=SU(2)/SU(1)$ and $\mathbb CP^n=U(n+1)/[U(1)\times U(n)]$, a map
$\varphi:S^3\rightarrow\mathbb CP^n$ is equivariant if there exist a Lie group
homomorphism $E:{SU(2)}\rightarrow{U(n+1)}$ and a holomorphic isometry
$A:\mathbb CP^n\rightarrow\mathbb CP^n$ such that $\varphi=A\circ\pi_1\circ E$, where
$\pi_1:{U(n+1)}\rightarrow\mathbb CP^n$ is the natural projection (cf. Definition
4.1 in \cite{Li}).

Following \cite{Be}, a submanifold $M\hookrightarrow\mathbb CP^n$ is called
a {\it CR-submanifold} if there is an orthogonal direct sum
decomposition $TM=V_1\oplus V_2$ such that $JV_1\subset T^\perp M$
and $JV_2=V_2$. In particular, if $V_2=\{0\}$ (resp. $V_1=\{0\}$),
then $M$ is called a {\it totally real} (resp. {\it complex})
submanifold.

From now on, we will use the following range convention of indices:
$$
A,B,\ldots=1,\ldots,n; \ \ i,j,\ldots=1,2,3;\ \
\alpha,\beta,\ldots=3,\ldots,n.
$$

In sequel of this section, we will present the fundamental structure
equations and several geometric conclusions for immersions from
$S^3$ into $\mathbb CP^n$ under the equivariant and CR type
conditions. This particularly includes Theorem \ref{thm:3.1}, which
states that an equivariant CR immersion
$\varphi:S^3\rightarrow\mathbb CP^n$ can not be totally geodesic.

\subsection{Basic formula for equivariant CR immersion $\varphi:S^3\rightarrow\mathbb CP^n$
}\label{sect:3.1}~

Let $\pi:S^{2n+1}\rightarrow\mathbb CP^n:v\mapsto\pi(v)=[v]$ be the Hopf
fibration. We suppose $\varphi:S^3\rightarrow\mathbb CP^n$ is an equivariant
immersion. According to Proposition 4.2 of \cite{Li}, there exists a
global unitary frame $\{e_0,e_A\}$ of the trivial bundle
$\underline{\mathbb C}^{n+1}=S^3\times \mathbb C^{n+1}$ such that $\varphi=[e_0]$ and
\begin{equation}\label{eqn:3.1}
\left\{
\begin{aligned}
de_0&=i\rho_0e_0+\sum\theta_A\ e_A,\\
de_A&=-\bar\theta_A e_0+\sum\theta_{AB}e_B,\ \ \theta_{AB}+\bar{\theta}_{BA}=0,
\end{aligned}
\right.
\end{equation}
where $\rho_0\in \mathfrak{su}(2)^*$, $\theta_A,\theta_{AB}\in
\mathfrak{su}(2)^*\otimes\mathbb C$. Hence, the induced metric
$ds^2=\sum\theta_A\bar\theta_A$ via $\varphi$ is left-invariant. Take an
orthonormal frame $\{\omega_i\}$, which is a basis of
$\mathfrak{su}(2)^*$, for $T^*S^3$. We can write $\theta_A=\sum
a_{Ai}\omega_i$, where $a_{Ai}\in\mathbb C$ are all constant. Then
\begin{equation}\label{eqn:3.2}
\sum\omega_i\otimes\omega_i=ds^2=\sum a_{Ai}\bar a_{Aj}\omega_i\omega_j
=\tfrac12\sum(a_{Ai}\bar a_{Aj}+\bar a_{Ai}a_{Aj})\omega_i\otimes\omega_j.
\end{equation}

Let $\Omega$ be the K\"ahler form of $\mathbb CP^n$, then we have
\begin{equation}\label{eqn:3.3}
\varphi^*\Omega=\tfrac
1{2\sqrt{-1}}\sum\theta_A\wedge\bar\theta_A=-\tfrac{\sqrt{-1}}2\sum
a_{Ai}\bar a_{Aj}\omega_i\wedge\omega_j =-\sum J_{ij}\omega_i\otimes\omega_j,
\end{equation}
where
\begin{equation}\label{eqn:3.4}
J_{ij}=\tfrac{\sqrt{-1}}2\sum(a_{Ai}\bar a_{Aj}-\bar
a_{Ai}a_{Aj})=-J_{ji}\in\mathbb{R}
\end{equation}
are also constant. The above equation together with \eqref{eqn:3.2}
gives
\begin{equation}\label{eqn:3.5}
\sum a_{Ai}\bar a_{Aj}=\delta_{ij}-\sqrt{-1}J_{ij}.
\end{equation}

The tensor $-\varphi^*\Omega$ determines a well defined bundle endomorphism
$$
{F=\sum J_{ij}\omega_i\otimes X_j}:{TS^3}\rightarrow{TS^3},
$$
where $\{X_i\}$ is the dual frame of $\{\omega_i\}$. Since $J_{ij}$'s
are constant, the rank of $F$ should be $0$ or $2$. If $F=0$, then
the submanifold $\varphi(S^3)$ is totally real (or called {\it weakly
Lagrangian} \cite{J-L}). If $F\ne0$, the tangent bundle $TS^3$
can be decomposed into orthogonal direct sum $TS^3=V_1\oplus V_2$
with $V_1=\ker F$. By re-choosing the frame we may assume $X_1\in
V_1$. Then
\begin{equation}\label{eqn:3.6}
FX_1=0,\ \ FX_2=J_{23}X_3,\ \ FX_3=-J_{23}X_2.
\end{equation}

Denote by $g$ the standard Fubini-Study metric of $\mathbb CP^n$ with
constant holomorphic sectional curvature $4$. From
$$
J_{23}=ds^2(FX_2,X_3)=-\varphi^*\Omega(X_2,X_3)=g(J\varphi_*(X_2),\varphi_*(X_3)),
$$
we see that $|J_{23}|\le1$. If $|J_{23}|=1$, then $\varphi(S^3)$ is
a CR-submanifold and $\varphi$ is said to be of {\it CR type}. In
fact, $FX_i$ is just the tangent component of $J\varphi_*(X_i)$ if
we identify $TS^3$ with $\varphi_*(TS^3)$.

Denote $\chi:=\arccos J_{23}$ the angle function. $\varphi(S^3)$ is a
totally real submanifold (resp. CR-submanifold) if and only if
$\chi=\pi/2$ (resp. $\chi=0$ or $\pi$). Therefore, the angle
function $\chi$ is an analogue of the K\"ahler angle of an immersed
surface in $\mathbb CP^n$.

From now on, we assume that the immersion $\varphi$ is of CR type. The
following lemma gives basic formulae which in sequel will be used
from time to time.

\begin{lemma}\label{lem:3.1}
Let $\varphi:S^3\rightarrow\mathbb CP^n$ be an equivariant CR immersion with
induced metric $ds^2$. Then there exist an orthonormal frame
$\{\omega_i\}$ of $T^*S^3$ and a unitary frame
$\{e_0,e_1,e_2,e_\alpha\}$ of $\underline{\mathbb C}^{n+1}$ such that $\varphi=[e_0]$, and
\begin{equation}\label{eqn:3.7}
\left\{
\begin{aligned}
&de_0=i\rho_0 e_0+\omega_1 e_1+\omega e_2,\\
&de_1=-\omega_1 e_0 + i\rho_1 e_1+\theta_{12} e_2+\sum\theta_{1\alpha} e_\alpha,\\
&de_2=-\bar\omega e_0 - \bar\theta_{12} e_1 +i\rho_2 e_2 +
\sum\theta_{2\alpha}e_\alpha,\\
&de_\alpha=- \bar\theta_{1\alpha}e_1-\bar\theta_{2\alpha}e_2 +
\sum\theta_{\alpha\beta}e_\beta, \ \ \theta_{\alpha\beta}+\bar{\theta}_{\beta\alpha}=0,
\end{aligned}
\right.
\end{equation}
where $\omega=\omega_2+i\omega_3$, $\rho_0,\,\rho_1,\,\rho_2\in
\mathfrak{su}(2)^*$, $\theta_{12},\,\theta_{1\alpha},\,\theta_{2\alpha}$,
$\theta_{\alpha\beta}\in \mathfrak{su}(2)^*\otimes\mathbb C$.
\end{lemma}

\begin{proof}
By assumption $J_{23}=\pm 1$ in \eqref{eqn:3.6}. This ensures us to
choose a left-invariant orthonormal frame $\{X_i\}$ of $TS^3$, by
replacing $X_2$ with $-X_2$ if necessary, such that
\begin{equation}\label{eqn:3.8}
FX_1=0,\ \ FX_2=X_3,\ \ FX_3=-X_2.
\end{equation}

On the other hand, as in \eqref{eqn:3.1} there exists a unitary fame
$\{e_0,e'_A\}$ of $\underline{\mathbb C}^{n+1}$ such that $\varphi=[e_0]$ and
\begin{equation}\label{eqn:3.9}
de_0=i\rho_0e_0+\sum \theta_A e'_A,\ \ de'_A=-\bar\theta_A e_0+\sum
\theta'_{AB}e'_B,
\end{equation}
where $\rho_0\in \mathfrak{su}(2)^*$, $\theta_A,\theta'_{AB}\in
\mathfrak{su}(2)^*\otimes\mathbb C$.

Let $\{\omega_i\}$ be the dual frame of $\{X_i\}$. We can write
$\theta_A=\sum a_{Ai}\omega_i$. Then \eqref{eqn:3.5} holds true with
$J_{12}=J_{13}=0$, $J_{23}=1$. The first equation of \eqref{eqn:3.9}
becomes
\begin{equation}\label{eqn:3.10}
de_0=i\rho_0e_0+\omega_1 f_1+\omega_2 f_2+\omega_3 f_3,
\end{equation}
where $f_i=\sum a_{Ai}e'_A$.

Denote by $\langle\cdot,\cdot\rangle$ the canonical symmetric scalar product
of $\mathbb C^{n+1}$. From \eqref{eqn:3.5} and \eqref{eqn:3.8} we get
$$
|f_1|=|f_2|=|f_3|=1,\ \ \langle{f_1},{\bar f_2}\rangle=\langle{f_1},{\bar
f_3}\rangle=0,\ \ \langle{f_2},{\bar f_3}\rangle=-i.
$$
This leads to $|f_2+if_3|^2=0$ and therefore $f_3=if_2$. Thus, by
\eqref{eqn:3.10},
\begin{equation}\label{eqn:3.11}
de_0=i\rho_0e_0+\omega_1 e_1+\omega e_2,
\end{equation}
where $\omega=\omega_2+i\omega_3$ and
\begin{equation}\label{eqn:3.12}
e_1=f_1=\sum\,a_{A1}e'_A, \ \ e_2=f_2=\sum\,a_{A2}e'_A.
\end{equation}

By \eqref{eqn:3.5} and \eqref{eqn:3.8}, we see that
$v_1=(a_{11},\ldots,a_{n1})$ and $v_2=(a_{12},\ldots,a_{n2})$ are
two orthogonal unit vectors of $\mathbb{C}^n$. Expand $\{v_1,v_2\}$
to be a unitary basis $\{v_1,v_2,v_3,\ldots,v_n\}$ of $\mathbb C^n$ such
that $v_\alpha=(v_{1\alpha},\ldots,v_{n\alpha}),\ \alpha=3,
\ldots,n$. Put
\begin{equation}\label{eqn:3.13}
e_3=\sum\,v_{A3}e'_A,\ \ldots,\ e_n=\sum\,v_{An}e'_A.
\end{equation}
Then $\{e_0,e_A\}$ forms a unitary frame of $\underline{\mathbb C}^{n+1}$. According
to \eqref{eqn:3.11}, the structure equations of the frame
$\{e_0,e_1,e_2,e_\alpha\}$ take the form of \eqref{eqn:3.7}. From
\eqref{eqn:3.9}, \eqref{eqn:3.12} and \eqref{eqn:3.13}, we know that
$\theta_{AB}=\langle{de_A},{\bar e_B}\rangle\in \mathfrak{su}(2)^*\otimes\mathbb C$
since $\theta'_{AB}\in \mathfrak{su}(2)^*\otimes\mathbb C$.
\end{proof}

\subsection{Invariants of equivariant CR immersion}\label{sect:3.2}~

Hereafter, $\{X_i\}$ is assumed the orthonormal frame satisfying
\eqref{eqn:3.8} and $\{\omega_i\}$ is its dual frame. Set
$Z=\frac12(X_2-iX_3)$ and $\omega=\omega_2+i\omega_3$. Then $\{X_1,Z,\bar
Z\}$ and $\{\omega_1,\omega,\bar\omega\}$ give a new pair of dual frames.
Obviously,
\begin{equation}\label{eqn:3.14}
X_2=Z+\bar Z,  \ \ X_3=i(Z-\bar Z).
\end{equation}

From \eqref{eqn:2.8}, with respect to $\{X_1,Z,\bar Z\}$ and
$\{\omega_1,\omega,\bar\omega\}$, we have
\begin{equation}\label{eqn:3.15}
\nabla X_1=-(\sigma Z+\bar\sigma\bar Z), \ \ \nabla Z=\tfrac12\bar\sigma
X_1-i\omega_{23}Z,
\end{equation}
where $\sigma=\omega_{12}+i\omega_{13}$, and by \eqref{eqn:2.10},
\begin{equation}\label{eqn:3.16}
\omega_{23}=(c_{23}^1+a)\omega_1-\tfrac i2\bar\mu\omega+\tfrac i2\mu\bar\omega,\
\ \sigma=\mu\omega_1+\tfrac i2c_{23}^1\omega-\tau\bar\omega,
\end{equation}
where
\begin{equation}\label{eqn:3.17}
\mu=c_{12}^1+ic_{13}^1,\ \ \tau=c_{21}^2+\tfrac
i2(c_{31}^2-c_{12}^3).
\end{equation}
Then, by \eqref{eqn:3.15} and \eqref{eqn:3.16}, we have
\begin{equation}\label{eqn:3.18}
\nabla_{X_1}X_1=-(\mu Z+\bar\mu\bar Z), \ \ \nabla_{\bar
Z}X_1=\tau Z+\tfrac i2 c_{23}^1\bar Z.
\end{equation}

The structure equation \eqref{eqn:2.6}, \eqref{eqn:2.7} can be
rewritten as
\begin{equation}\label{eqn:3.19}
\left\{
\begin{aligned}
&d\omega_1=-\tfrac12[\bar\mu\omega_1\wedge\omega+\mu\omega_1\wedge\bar\omega+ic_{23}^1\omega\wedge\bar\omega],\\
&d\omega=\tfrac
i2(c_{23}^1+2a)\omega_1\wedge\omega+\tau\omega_1\wedge\bar\omega+\tfrac12\mu\omega\wedge\bar\omega,
\end{aligned}
\right.
\end{equation}
\begin{equation}\label{eqn:3.20}
d\omega_{23}=\tfrac{i}2\sigma\wedge\bar{\sigma}+\Omega_{23},\ \
d\sigma=-i\sigma\wedge\omega_{23}+\Omega_{12}+i\Omega_{13}.
\end{equation}

With respect to $\{\omega_1,\omega,\bar\omega\}$, the curvature forms can be expressed as
$$
\Omega_{ij}=\tfrac12[(R_{ij12}-iR_{ij13})\omega_1\wedge\omega+(R_{ij12}+iR_{ij13})\omega_1\wedge\bar\omega+iR_{ij23}\omega\wedge\bar\omega].
$$
It follows that
\begin{equation}\label{eqn:3.21}
\Omega_{23}=\tfrac12[(R_{2312}-iR_{2313})\omega_1\wedge\omega+(R_{2312}+iR_{2313})\omega_1\wedge\bar\omega
+iR_{2323}\omega\wedge\bar{\omega}],
\end{equation}
\begin{equation}\label{eqn:3.22}
\begin{aligned}
\Omega_{12}+i\Omega_{13}=\tfrac12[&(R_{1212}+R_{1313})\omega_1\wedge\omega
+(R_{1212}-R_{1313}+2iR_{1213})\omega_1\wedge\bar\omega\\
&+(iR_{1223}-R_{1323})\omega\wedge\bar{\omega}].
\end{aligned}
\end{equation}

For any orthonormal frame $\{\tilde X_i\}$ with $F\tilde{X}_1=0$ and $F\tilde X_2=\tilde X_3$ we must have
$$
\tilde X_1=(\det T) X_1,\ \ \tilde X_2=\cos t\,X_2-\sin t\,X_3,\ \ \tilde
X_3=\sin t\,X_2+\cos t\,X_3,
$$
where $t$ is a real number, and $T$ is the transition matrix from
$\{ X_i\}$ to $\{\tilde X_i\}$. If we set $Z'=\tfrac12(\tilde X_2-i\tilde
X_3)$, then the above frame transformation can be written as
\begin{equation}\label{eqn:3.23}
\tilde{X}_1=\varepsilon X_1,\ \ Z'=e^{-it}Z,\ \ \bar Z'=e^{it}\bar Z
\end{equation}
where $\varepsilon=\det T$.

Now we prove the following proposition.
\begin{proposition}\label{prop:3.1}
Let $\varphi:S^3\rightarrow\mathbb CP^n$ be an equivariant CR
immersion with induced metric $ds^2$. Then we have:

{\rm (i)} $c_{12}^1=c_{13}^1=0$ if and only if $\nabla_{X_1}X_1=0$,
that is, the integral curve of $X_1$ is a geodesic of
$(S^3,ds^2)$;

{\rm (ii)} $c_{12}^1=c_{13}^1=c_{21}^2=0$ and $c_{31}^2=c_{12}^3$ if and only if $(S^3,ds^2)$ is a Berger
sphere;

{\rm (iii)} Under the frame transformation \eqref{eqn:3.23}, the
structure constants $|c_{12}^1+ic_{13}^1|$, $|c_{23}^1|$ and
$|c_{21}^2+\tfrac{i}2(c_{31}^2-c_{12}^3)|$ are all invariant.
\end{proposition}

\begin{proof} The assertion (i) is from the first equation of \eqref{eqn:3.18}.
The assertion (ii) is a direct consequence of Proposition
\ref{prop:2.2}, \eqref{eqn:3.17} and \eqref{eqn:3.19}. Finally, we
obtain from \eqref{eqn:3.18} that
$$
c_{23}^1=-4ids^2(\nabla_{\bar Z}X_1,Z),\ \mu=-2ds^2(\nabla_{X_1}X_1,\bar Z),\ \tau=2ds^2(\nabla_{\bar Z}X_1,\bar Z).
$$

The assertion (iii) follows from the above computation of
$c_{23}^1,\,\mu$ and $\tau$.
\end{proof}

\subsection{The Gauss-Weingarten formulae }\label{sect:3.3}~

Let $\widetilde{\mathbb{C}}^{n+1}=\mathbb CP^n\times\mathbb C^{n+1}$
be the trivial bundle, $L$ the canonical line bundle over $\mathbb
CP^n$. Denote by $L^{\bot}$ (resp. $\bar{L}$) the orthogonal
complement (resp. the complex conjugate) of $L$ in
$\widetilde{\mathbb{C}}^{n+1}$. The trivial connection $d$ of
$\widetilde{\mathbb{C}}^{n+1}$ induces the standard connection on
both $L$ and $L^{\bot}$. This produces the connection on $\bar L$
and hence on $\bar L\otimes L^{\bot}$.

To clarify the above let us take a local unitary frame $\{e_0,\,
e_A\}$ of $\widetilde{\mathbb{C}}^{n+1}$ with $e_0\in L$. For
simplicity, we will denote by the same notation $\nabla$ the
connections on different vector bundles. Let us write
$$
de_0=i\rho_0e_0+\sum \theta_A e_A,\ \ de_A=-\bar\theta_A e_0+\sum
\theta_{AB}e_B.
$$
Then by definition
$$
\nabla e_0=i\rho_0e_0,\ \ \nabla e_A=\sum \theta_{AB}e_B.
$$
Thus $\nabla\bar e_0=-i\rho_0\bar e_0$, and consequently,
\begin{equation}\label{eqn:3.24}
\nabla(\bar e_0\otimes e_A)={\textstyle\sum}(\theta_{AB}-i\delta_{AB}\rho_0)(\bar{e}_0\otimes e_B).
\end{equation}

If we identify $T\mathbb CP^n\cong_\mathbb{R} T^{(1,0)}\mathbb CP^n$ with $\bar L\otimes
L^{\bot}$, then the connection on $\bar L\otimes L^{\bot}$ is
exactly the Levi-Civita connection on $T\mathbb CP^n$. With respect to the
unitary frame $\{\bar e_0\otimes e_A\}$ of $\bar L\otimes L^{\bot}$,
$\{\theta_{AB}-i\delta_{AB}\rho_0\}$ are the connection $1$-forms.

For an immersion $\varphi:{S^3}\rightarrow{\mathbb CP^n}$, the pullback bundle of
$\widetilde{\mathbb{C}}^{n+1}$ via $\varphi$ will be denoted also by
$\underline{\mathbb C}^{n+1}$, whereas those of $L$, $\bar L$ and $L^{\bot}$ via
$\varphi$ will still be denoted by the same notations, respectively.
Using Lemma \ref{lem:3.1} we have a global unitary frame
$\{e_0,e_A\}$ of $\underline{\mathbb C}^{n+1}$ with $\varphi=[e_0]$ such that
\eqref{eqn:3.7} holds. Then the differential (tangent map) of $\varphi$,
namely that
$$
d\varphi:{TS^3}\rightarrow{T\mathbb CP^n\cong_\mathbb{R} T^{(1,0)}\mathbb CP^n=\bar L\otimes
L^{\bot}},
$$
is given by $d\varphi=\omega_1(\bar e_0\otimes e_1)+\omega(\bar e_0\otimes
e_2)$, or equivalently,
\begin{equation}\label{eqn:3.25}
\varphi_*X_1=\bar e_0\otimes e_1,\ \ \varphi_*X_2=\bar e_0\otimes e_2,\ \
\varphi_*X_3=i\bar e_0\otimes e_2.
\end{equation}
Then, the tangent bundle and normal bundle of $\varphi(S^3)$ are
given, respectively, by
\begin{equation}\label{eqn:3.26}
\begin{aligned}
&\varphi_*(TS^3)={\rm span}_\mathbb{R}\{\bar e_0\otimes e_1,\bar e_0\otimes
e_2,i\bar e_0\otimes e_2\}\subset\varphi^{-1}T\mathbb CP^n,\\
&T^{\bot}S^3={\rm span}_\mathbb{R}\{i\bar e_0\otimes e_1,\bar e_0\otimes
e_\alpha,i\bar e_0\otimes e_\alpha\}\subset\varphi^{-1}T\mathbb CP^n.
\end{aligned}
\end{equation}

Denote by $\bar\nabla$ the connection on the pullback bundle
$\varphi^{-1}T\mathbb CP^n$. From \eqref{eqn:3.24} and \eqref{eqn:3.25},
together with \eqref{eqn:3.7}, we see that
\begin{equation}\label{eqn:3.27}
\left\{
\begin{aligned}
&\bar\nabla(\varphi_*X_1)=i(\rho_1-\rho_0)(\bar e_0\otimes
e_1)+\theta_{12}(\bar e_0\otimes e_2)+
\sum\theta_{1\beta}(\bar e_0\otimes e_\beta),\\
&\bar\nabla(\varphi_*X_2)=-\bar\theta_{12}(\bar e_0\otimes
e_1)+i(\rho_2-\rho_0)(\bar e_0\otimes e_2)
+\sum\theta_{2\beta}(\bar e_0\otimes e_\beta),\\
&\bar\nabla(\varphi_*X_3)=-i\bar\theta_{12}(\bar e_0\otimes
e_1)-(\rho_2-\rho_0)(\bar e_0\otimes e_2)+i\sum\theta_{2\beta}(\bar
e_0\otimes e_\beta).
\end{aligned}
\right.
\end{equation}

Denote by $B$ the second fundamental form of $\varphi$. Then the Gauss
formula
$$
\bar\nabla_{X_j}(\varphi_*X_i)=\varphi_*(\nabla_{X_j} X_i)+B(X_i,X_j)
$$
implies that the tangential component of the right hand side of
\eqref{eqn:3.27} is $\varphi_*(\nabla X_j)$.

From \eqref{eqn:3.15} and \eqref{eqn:3.25}, we get
$$
\theta_{12}(\bar e_0\otimes e_2)=\varphi_*(\nabla X_1)=-\sigma(\bar e_0\otimes e_2).
$$
It follows that
\begin{equation}\label{eqn:3.28}
\theta_{12}=-\sigma=-(\omega_{12}+i\omega_{13}).
\end{equation}

Similarly, the above fact, together with \eqref{eqn:3.25} and
$\nabla X_2=\omega_{12}X_1-\omega_{23}X_3$, gives
$$
i(\rho_2-\rho_0)(\bar e_0\otimes e_2)=\varphi_*(\nabla X_2-\omega_{12}X_1)=-i\omega_{23}(\bar e_0\otimes e_2).
$$
Thus we get
\begin{equation}\label{eqn:3.29}
\rho_0-\rho_2=\omega_{23}.
\end{equation}

Note that the second fundamental form can be expressed by
$$
B=\sum[\bar\nabla(\varphi_*X_i)-\varphi_*(\nabla X_i)]\otimes\omega_i.
$$
By using \eqref{eqn:3.27} and \eqref{eqn:3.28}, we have
\begin{equation}\label{eqn:3.30}
\begin{aligned}
B=[(\rho_1-\rho_0)\otimes\omega_1&-\omega_{13}\otimes\omega_2+\omega_{12}\otimes\omega_3](i\,\bar
e_0\otimes e_1) \\
&+\sum(\theta_{1\alpha}\otimes\omega_1+\theta_{2\alpha}\otimes\omega)(\bar e_0\otimes
e_\alpha).
\end{aligned}
\end{equation}

Since, by \eqref{eqn:3.16}, $\sigma=\omega_{12}+i\omega_{13}\ne0$, and that
$\det C\not=0$ implies that $\{c_{ij}^1\}$ can not be all zero, it
follows that $B\ne0$. Hence we have
\begin{theorem}\label{thm:3.1}
An equivariant CR immersion $\varphi:S^3\rightarrow\mathbb CP^n$ can not be totally
geodesic.
\end{theorem}

From \eqref{eqn:3.26} we see that the set of normal vector fields:
\begin{equation}\label{eqn:3.31}
\{\xi_0:=i\bar{e}_0\otimes e_1=J(\varphi_*X_1),\
\xi_\alpha:=\bar{e}_0\otimes e_\alpha,\
J\xi_\alpha=i\bar{e}_0\otimes e_\alpha\}
\end{equation}
defines an orthonormal frame of $T^{\bot}S^3$. Moreover, by
\eqref{eqn:3.24}, we have
$$
\begin{aligned}
\bar\nabla\xi_0&=(\rho_0-\rho_1)(\bar e_0\otimes e_1)+i\theta_{12}(\bar e_0\otimes e_2)+i{\textstyle\sum}\theta_{1\alpha}(\bar e_0\otimes e_\alpha),\\
\bar\nabla\xi_\alpha&=-\bar\theta_{1\alpha}\bar e_0\otimes e_1-\bar\theta_{2\alpha}\bar e_0\otimes e_2
+{\textstyle\sum}(\theta_{\alpha\beta}-i\delta_{\alpha\beta}\rho_0)(\bar e_0\otimes e_\beta).
\end{aligned}
$$
Then, according to the Weingarten formula
$$
\bar\nabla\xi=-\varphi_*\circ A_\xi+\nabla^\bot\xi,\ \ \xi\in
T^\bot S^3,
$$
we get
\begin{equation}\label{eqn:3.32}
\nabla^\bot\xi_0=\tfrac i2\sum(\theta_{1\alpha}-\bar\theta_{1\alpha})\xi_\alpha
        +\tfrac12\sum(\theta_{1\alpha}+\bar\theta_{1\alpha})J\xi_\alpha,
\end{equation}
$$
\begin{aligned}
&A_{\xi_\alpha}=\tfrac12\sum(\theta_{1\alpha}+\bar\theta_{1\alpha})X_1+\tfrac12\sum
        (\theta_{2\alpha}+\bar\theta_{2\alpha})X_2-\tfrac i2\sum(\theta_{2\alpha}-\bar\theta_{2\alpha})X_3, \\
&A_{J\xi_\alpha}=-\tfrac i2\sum(\theta_{1\alpha}-\bar\theta_{1\alpha})X_1+\tfrac
        i2\sum(\theta_{2\alpha}-\bar\theta_{2\alpha})X_2+\tfrac12\sum(\theta_{2\alpha}+\bar\theta_{2\alpha})X_3.
\end{aligned}
$$

Put $W_1:={\rm span}\{\xi_0\}$ and $W_2:=W_1^\perp$ in
$T^{\bot}S^3$. If $A_\xi=0$ for any $\xi\in W_2$, the submanifold
$\varphi(S^3)$ is said to be {\it totally geodesic with respect to
$W_2$}. Then, from the above calculations, we have the following
observations:
\begin{proposition}\label{prop:3.2}
{\rm(i)} $\xi_0$ is parallel in $T^{\bot}S^3$ $($that is
$\nabla^\perp\xi_0=0)$ if and only if $\theta_{1\alpha}=0$. {\rm(ii)} The
submanifold $\varphi(S^3)$ is totally geodesic with respect to $W_2$
if and only if $\theta_{1\alpha}=\theta_{2\alpha}=0$.
\end{proposition}

\subsection{The Gauss-Codazzi equations}\label{sect:3.4}~

Exterior differentiation of \eqref{eqn:3.7}, and using
\eqref{eqn:3.28}, leads to the following eqs.:
\begin{equation}\label{eqn:3.33}
id\rho_0=-\omega\wedge\bar\omega=2i\omega_2\wedge\omega_3,
\end{equation}
\begin{equation}\label{eqn:3.34}
d\omega_1=i(\rho_0-\rho_1)\wedge\omega_1-\bar\sigma\wedge\omega,\ \
\theta_{1\alpha}\wedge\omega_1+\theta_{2\alpha}\wedge\omega=0,
\end{equation}
\begin{equation}\label{eqn:3.35}
d\omega=\sigma\wedge\omega_1+i(\rho_0-\rho_2)\wedge\omega,
\end{equation}
\begin{equation}\label{eqn:3.36}
id\rho_1=-\sigma\wedge\bar\sigma-\sum\theta_{1\alpha}\wedge\bar\theta_{1\alpha},\ \
id\rho_2=\omega\wedge\bar\omega+\sigma\wedge\bar\sigma-\sum\theta_{2\alpha}\wedge\bar\theta_{2\alpha},
\end{equation}
\begin{equation}\label{eqn:3.37}
d\sigma=\omega_1\wedge\omega+i(\rho_1-\rho_2)\wedge\sigma+\sum\theta_{1\alpha}\wedge\bar\theta_{2\alpha},
\end{equation}
\begin{equation}\label{eqn:3.38}
\left\{
\begin{aligned}
d\theta_{1\alpha}&=i\rho_1\wedge\theta_{1\alpha}-\sigma\wedge\theta_{2\alpha}+\sum\theta_{1\beta}\wedge\theta_{\beta\alpha},\\
d\theta_{2\alpha}&=i\rho_2\wedge\theta_{2\alpha}+\bar\sigma\wedge\theta_{1\alpha}+\sum\theta_{2\beta}\wedge\theta_{\beta\alpha},
\end{aligned}
\right.
\end{equation}
\begin{equation}\label{eqn:3.39}
d\theta_{\alpha\beta}=\theta_{\alpha1}\wedge\theta_{1\beta}+\theta_{\alpha2}\wedge\theta_{2\beta}+\sum\theta_{\alpha\gamma}\wedge\theta_{\gamma\beta}.
\end{equation}

From \eqref{eqn:3.30} we see that the symmetry of $B$ is given by
\eqref{eqn:3.34}. In fact, by using \eqref{eqn:2.6}, the first
equation of \eqref{eqn:3.34} means that
$$
(\rho_1-\rho_0)\wedge\omega_1=i(d\omega_1+\bar\sigma\wedge\omega)=\omega_{13}\wedge\omega_2-\omega_{12}\wedge\omega_3.
$$
Substituting \eqref{eqn:2.10} into the above equation, and using
\eqref{eqn:2.16}, we get
\begin{equation}\label{eqn:3.40}
(\rho_1-\rho_0)\wedge\omega_1=(c_{23}^2\omega_2+c_{23}^3\omega_3)\wedge\omega_1.
\end{equation}

The exterior differentiation of \eqref{eqn:3.29}, together with
\eqref{eqn:3.33}, \eqref{eqn:3.36}, \eqref{eqn:3.37} and
\eqref{eqn:3.20}, gives the Gauss equation
\begin{equation}\label{eqn:3.41}
\left\{
\begin{array}{@{\extracolsep{3pt}}r@{\extracolsep{3pt}}c@{\extracolsep{3pt}}l}
\Omega_{23}=&d\omega_{23}-\tfrac
i2\sigma\wedge\bar\sigma&=2i\omega\wedge\bar\omega+\frac i2\sigma\wedge\bar\sigma
-i\sum\theta_{2\alpha}\wedge\bar\theta_{2\alpha},\\[3pt]
\Omega_{12}+i\Omega_{13}=&d\sigma+i\sigma\wedge\omega_{23}&=i(\rho_1-\rho_0)\wedge\sigma+\omega_1\wedge\omega
+\sum\theta_{1\alpha}\wedge\bar\theta_{2\alpha}.
\end{array}
\right.
\end{equation}

Similarly, from \eqref{eqn:3.27}, noting that $\sigma=-\theta_{12}$, we
see that \eqref{eqn:3.37} and \eqref{eqn:3.38} are part of the
Codazzi equations, whereas the rest one is
$$
d(\rho_0-\rho_1)=i(\omega\wedge\bar\omega-\sigma\wedge\bar\sigma-{\textstyle\sum}\theta_{1\alpha}\wedge\bar\theta_{1\alpha}).
$$

\section{Equivariant CR minimal immersions from $S^3$ into $\mathbb CP^n$}\label{sect:4}

The purpose of this section is the proof of Theorem \ref{thm:4.1},
which concerns with a roughly classification of equivariant CR
minimal immersions from $S^3$ into $\mathbb CP^n$. Recall that a
equivariant CR immersion $\varphi:S^3\rightarrow\mathbb CP^n$ is minimal if and
only if ${\rm trace}\,B=0$, that is, by \eqref{eqn:3.30},
$$
\left\{
\begin{aligned}
&{\rm trace}\,[(\rho_1-\rho_0)\otimes\omega_1-\omega_{13}\otimes\omega_2+\omega_{12}\otimes\omega_3]=0,\\
&{\rm trace}\,(\theta_{1\alpha}\otimes\omega_1+\theta_{2\alpha}\otimes\omega)=0.
\end{aligned}
\right.
$$
This, by the use of \eqref{eqn:2.10} and \eqref{eqn:2.11}, is
equivalent to
\begin{equation}\label{eqn:4.1}
(\rho_1-\rho_0)(X_1)=\omega_{13}(X_2)-\omega_{12}(X_3)=c_{23}^1,\ \
\theta_{1\alpha}(X_1)+2\theta_{2\alpha}(\bar Z)=0.
\end{equation}

The above two equations, together with \eqref{eqn:3.40},
\eqref{eqn:2.10} and the second equation of \eqref{eqn:3.34}, lead
to
\begin{equation}\label{eqn:4.2}
\rho_1-\rho_0=c_{23}^1\omega_1+c_{23}^2\omega_2+c_{23}^3\omega_3=\omega_{23}-a\omega_1,
\end{equation}
\begin{equation}\label{eqn:4.3}
\theta_{1\alpha}=\lambda_\alpha\omega,\ \ \theta_{2\alpha}=\lambda_\alpha\omega_1+\mu_\alpha\omega.
\end{equation}
where $\lambda_\alpha,\;\mu_\alpha\in\mathbb C$ are constant.

Recall that a submanifold $M\hookrightarrow\mathbb CP^n$ is called {\it
linearly full} if there is no totally geodesic $\mathbb C P^k\;(k<n)$
such that $M\hookrightarrow\mathbb C P^k\subset\mathbb CP^n$. For such a
submanifold the dimension $n$ is called the {\it full dimension}.
Henceforth, we assume that $\varphi:S^3\rightarrow\mathbb CP^n$ is an equivariant
CR minimal immersion with full dimension $n$.

First of all, as the continuation of Proposition \ref{prop:3.2}
(ii), we have

\begin{lemma}\label{lem:4.1}
{\rm (i)} If $\varphi:S^3\rightarrow\mathbb CP^n$ is an equivariant CR immersion
with the full dimension $n$ and that $\varphi(S^3)$ is totally
geodesic with respect to $W_2$, then $n=2$.

{\rm (ii)} If $\varphi:S^3\to\mathbb{C}P^2$ is an equivariant CR
minimal immersion, then the structure constants $\{c_{ij}^k\}$ with
respect to an orthonormal basis $\{X_1,X_2,X_3\}$ and coefficients
forms of equation \eqref{eqn:3.7} satisfy:
\begin{equation}\label{eqn:4.4}
\rho_0=\tfrac a3\omega_1,\ \rho_1=(\tfrac a3+c_{23}^1)\omega_1,\
\rho_2=-(\tfrac{2a}3+c_{23}^1)\omega_1, \ \sigma=\tfrac{i}2
c_{23}^1\omega-\tau\bar{\omega},
\end{equation}
\begin{equation}\label{eqn:4.5}
c_{12}^1=c_{13}^1=0,\ \ ac_{23}^1=-6,\ \
4\tau\bar\tau=3(c_{23}^1)^2-4,\ \ \tau(3c_{23}^1+2a)=0.
\end{equation}
\end{lemma}

\begin{proof}
(i) If $\varphi(S^3)$ is totally geodesic with respect to $W_2$,
then, from Proposition \ref{prop:3.2}, we have
$\theta_{1\alpha}=\theta_{2\alpha}=0$. Since $\theta_{12}=-\sigma$, the first
three equations in \eqref{eqn:3.7} become
\begin{equation}\label{eqn:4.6}
\left\{
\begin{aligned}
de_0&=i \rho_0  e_0 + \omega_1 e_1 + \omega e_2,\\
de_1&=-\omega_1 e_0+ i\rho_1 e_1 - \sigma e_2,\\
de_2&=- \bar\omega e_0 + \bar\sigma e_1 +i\rho_2 e_2.
\end{aligned}
\right.
\end{equation}
Then by \eqref{eqn:3.33} and \eqref{eqn:3.36} we get
$d(\rho_0+\rho_1+\rho_2)=0$. It follows by the isomorphism of
$d:\mathfrak{su}(2)^*\to{\mathfrak{su}(2)^*\wedge
\mathfrak{su}(2)^*}$ that $\rho_0+\rho_1+\rho_2=0$. Then we get
$d(e_0\wedge e_1\wedge e_2)=0$, showing that ${\it
span}\{e_0,e_1,e_2\}$ is a constant subspace. The fullness condition
thus implies that $n=2$.

(ii) In this situation, as above we still have
$\rho_0+\rho_1+\rho_2=0$. Thus, by \eqref{eqn:3.29} and
\eqref{eqn:4.2}, it holds
\begin{equation}\label{eqn:4.7}
3\rho_0=(\rho_0-\rho_2)-(\rho_1-\rho_0)=a\omega_1.
\end{equation}

Taking exterior differentiation of \eqref{eqn:4.7}, together with
the use of \eqref{eqn:3.33} and the first equation of
\eqref{eqn:3.19}, gives
\begin{equation}\label{eqn:4.8}
\mu=c_{12}^1+ic_{13}^1=0,\ \ ac_{23}^1=-6.
\end{equation}
Then, by \eqref{eqn:3.16}, we have
\begin{equation}\label{eqn:4.9}
\omega_{23}=(c_{23}^1+a)\omega_1, \ \ \sigma=\tfrac
i2c_{23}^1\omega-\tau\bar{\omega},\ \
\sigma\wedge\bar\sigma=\tfrac14[(c_{23}^1)^2-4\tau\bar\tau]\omega\wedge\bar\omega.
\end{equation}

Since $\mu=0$, from \eqref{eqn:3.17} and \eqref{eqn:2.16}, we see
that $c_{23}^2=c_{23}^3=0$. Then, by \eqref{eqn:4.2},
\eqref{eqn:3.29} and that $\rho_0=(a/3)\omega_1$, we get
\begin{equation}\label{eqn:4.10}
\rho_1=\rho_0+c_{23}^1\omega_1=(\tfrac a3+c_{23}^1)\omega_1,\ \
\rho_2=\rho_0-\omega_{23}=-(\tfrac{2a}3+c_{23}^1)\omega_1.
\end{equation}
This completes the proof of \eqref{eqn:4.4}.

Next, by \eqref{eqn:3.19}, $d\omega_1=-\tfrac
i2c_{23}^1\omega\wedge\bar\omega$. Then, using $ac_{23}^1=-6$,
\eqref{eqn:4.10}, \eqref{eqn:3.36} and \eqref{eqn:4.9}, we get
$$
-\tfrac i2c_{23}^1(\tfrac
a3+c_{23}^1)\omega\wedge\bar\omega=d\rho_1=i\sigma\wedge\bar\sigma=\tfrac
i4[(c_{23}^1)^2-4\tau\bar\tau]\omega\wedge\bar\omega,
$$
and it follows that $4\tau\bar\tau=3(c_{23}^1)^2-4$.

Finally, by using \eqref{eqn:4.4} and \eqref{eqn:4.9}, the second
equation of \eqref{eqn:3.41} becomes
$$
\Omega_{12}+i\Omega_{13}=[1-\tfrac12(c_{23}^1)^2]\omega_1\wedge\omega-i\tau c_{23}^1\omega_1\wedge\bar{\omega}.
$$
This, if compared it with \eqref{eqn:3.22}, gives
$$
R_{1212}-R_{1313}+2iR_{1213}=-2i\tau c_{23}^1.
$$

On the other hand, using $\mu=0$ and \eqref{eqn:2.16}, from
\eqref{eqn:2.20} we obtain
$$
R_{1212}-R_{1313}+2iR_{1213}=(c_{12}^3-c_{31}^2+2ic_{21}^2)(2c_{23}^1+2a)=2i{\tau}(2c_{23}^1+2a).
$$
Then the last equation of \eqref{eqn:4.5} follows.
\end{proof}

\begin{corollary}\label{cor:4.1}
Let $\varphi:S^3\rightarrow\mathbb{C}P^2$ be an equivariant CR minimal
immersion.

{\rm (i)} If $\varphi(S^3)$ is not a Berger sphere, then there is an
orthonormal frame $\{\omega_i\}$ such that, in \eqref{eqn:4.6}, it
holds
\begin{equation}\label{eqn:4.11}
\rho_0=\omega_1,\ \ \rho_1=-\omega_1,\ \ \rho_2=0,\ \
\sigma=-i\omega-\sqrt2\bar\omega.
\end{equation}

{\rm (ii)} If $\varphi(S^3)$ is a Berger sphere, then there is an
orthonormal frame $\{\omega_i\}$ such that, in \eqref{eqn:4.6},
\begin{equation}\label{eqn:4.12}
c_{23}^1=-\tfrac2{\sqrt3},\ c_{31}^2=c_{12}^3=-\tfrac8{\sqrt3},\
\rho_0=\sqrt{3}\omega_1, \ \rho_1=\tfrac1{\sqrt{3}}\omega_1,\ \sigma=-\tfrac
i{\sqrt{3}}\omega.
\end{equation}
\end{corollary}
\begin{proof}
(i) By \eqref{eqn:4.5} it holds $c_{12}^1=c_{13}^1=0$. Then, it
follows from Proposition \ref{prop:3.1} that $\varphi(S^3)$ is not a
Berger sphere if and only if $\tau=c_{21}^2+\tfrac
i2(c_{31}^2-c_{12}^3)\not=0$. Then, by the third conclusion of
Proposition \ref{prop:3.1}, we can take a suitable frame
transformation \eqref{eqn:3.23} such that $c_{23}^1<0$ and
$\tau=|\tau|>0$. It follows from \eqref{eqn:4.5} that $c_{23}^1=-2,\
a=3$ and $\tau=\sqrt{2}$. Substituting these into \eqref{eqn:4.4} we
obtain \eqref{eqn:4.11}.

(ii) By replacing $X_1$ with $-X_1$ if necessary, we may assume that
$c_{23}^1<0$. Then by \eqref{eqn:4.5} and $\tau=0$ we have
$c_{23}^1=-2/\sqrt3$, $a=3\sqrt{3}$. Put these into \eqref{eqn:4.4}
we get \eqref{eqn:4.12}.
\end{proof}

The following proposition shows that $n=2$ is very exceptional, only
in that case a non-Berger sphere could be admitted.

\begin{proposition}\label{prop:4.1}
Any equivariant CR minimal
immersion $\varphi:S^3\rightarrow\mathbb CP^n$ with $n\ge3$ must be a Berger
sphere.
\end{proposition}

\begin{proof} We consider two cases for $n\ge3$:\ \ (a) $\nabla^\bot\xi_0=0$;
\ \ (b) $\nabla^\bot\xi_0\ne0$.

(a)\ \ If $\nabla^\bot\xi_0=0$, then according to Proposition
\ref{prop:3.2} and Lemma \ref{lem:4.1}, we have $\lambda_\alpha=0$
and $\sum|\mu_\alpha|^2>0$ in \eqref{eqn:4.3}. By re-choosing $e_\alpha$'s
in the frame $\{e_0,e_1,e_2,e_\alpha\}$ as stated in Lemma
\ref{lem:3.1} we may assume that
\begin{equation}\label{eqn:4.13}
\theta_{1\alpha}=0,\ \ \theta_{23}=\mu_3\omega,\ \ \theta_{24}=\cdots=\theta_{2n}=0,
\end{equation}
where $\mu_3\in\mathbb{R}$ is positive. By \eqref{eqn:3.38} we have
$\sigma\wedge\omega=0$, and thus by \eqref{eqn:3.16} it holds
$\mu=c_{12}^1+ic_{13}^1=0$ and $\tau=c_{21}^2+\tfrac
i2(c_{31}^2-c_{12}^3)=0$. It follows from Proposition \ref{prop:3.1}
(ii) that $\varphi(S^3)$ is a Berger sphere.

(b)\ \ In this case,  Proposition \ref{prop:3.2} implies that
$\sum|\lambda_\alpha|^2>0$. Then, by re-choosing $e_\alpha$'s of
$\{e_0,e_1,e_2,e_\alpha\}$ we have constants $\lambda_3>0$ and $\mu_4\ge0$
such that
\begin{equation}\label{eqn:4.14}
\left\{
\begin{aligned}
&\theta_{13}=\lambda_3\omega,\ \ \theta_{14}=\cdots=\theta_{1n}=0,\\
&\theta_{23}=\lambda_3\omega_1+\mu_{3}\omega,\ \ \theta_{24}=\mu_4\omega,\ \
\theta_{25}=\cdots=\theta_{2n}=0.
\end{aligned}
\right.
\end{equation}

Using \eqref{eqn:3.35} and \eqref{eqn:3.38} we have
$$
\lambda_3\sigma\wedge\omega_1+i\lambda_3(\rho_0-\rho_2)\wedge\omega=d\theta_{13}=i\lambda_3(\rho_1-\rho_3)\wedge\omega-\sigma\wedge(\lambda_3\omega_1+\mu_3\omega),
$$
where $i\rho_3=\theta_{33}$. Using \eqref{eqn:3.16} and comparing
the coefficient of $\omega_1\wedge\bar\omega$ on both side of the
above equation, we get $\tau=c_{21}^2+\tfrac
i2(c_{31}^2-c_{12}^3)=0$. Thus $c_{21}^2=-c_{31}^3=0$,
$c_{31}^2=c_{12}^3$.

It follows by \eqref{eqn:3.16} and \eqref{eqn:4.2} that
$$
\begin{array}{c}
\omega_{23}=(c_{23}^1+a)\omega_1-\tfrac i2\bar\mu\omega+\tfrac i2\mu\bar\omega,\ \ \sigma=\mu\omega_1+\tfrac{i}2c_{23}^1\omega, \\[2pt]
\rho_1-\rho_0=\omega_{23}-a\omega_1=c_{23}^1\omega_1-\tfrac i2\bar\mu\omega+\tfrac
i2\mu\bar\omega.
\end{array}
$$

Substituting these into the second equation of \eqref{eqn:3.41} and
then comparing the coefficient of $\omega_1\wedge\bar\omega$ with that in
\eqref{eqn:3.22}, we obtain
$$
R_{1212}-R_{1313}+2iR_{1213}=\mu^2=(c_{12}^1)^2+(c_{13}^1)^2+2ic_{12}^1c_{13}^1.
$$

On the other hand, from \eqref{eqn:2.20}, and using \eqref{eqn:3.17}
with the fact $\tau=0$, we get
$$
R_{1213}=-2c_{12}^1c_{13}^1, \quad
R_{1212}-R_{1313}=2(c_{13}^1)^2-2(c_{12}^1)^2.
$$
Then, combining with the previous equation, it implies that
$c_{12}^1=c_{13}^1=0$.

Thus, as in case (a), $\varphi(S^3)\subset\mathbb CP^n$ is a Berger
sphere.
\end{proof}

The following lemma is also needed in Sect. \ref{sect:6}.
\begin{lemma}\label{lem:4.2}
Let $\varphi:S^3\rightarrow\mathbb CP^n$ be an equivariant CR minimal immersion
with $n\ge3$. Then, there exist positive real numbers $b$ and $c$, a
normalized frame $\{\omega_i\}$ and a unitary frame $\{e_0,e_A\}$ of
$\underline{\mathbb{C}}^{n+1}$ such that $\varphi=[e_0]$ and
\begin{equation}\label{eqn:4.15}
c_{23}^1=-2b\sqrt{c},\ \ c_{31}^2=c_{12}^3=-2\sqrt{c}/b,
\end{equation}
\begin{equation}\label{eqn:4.16}
\left\{
\begin{aligned}
&de_0=\tfrac i{b\sqrt c}\omega_1e_0+\omega_1e_1+\omega\,e_2,\\
&de_1=-\omega_1e_0+\tfrac{i(1-2b^2c)}{b\sqrt
c}\omega_1e_1+ib\sqrt{c}\omega\,e_2+ \lambda_3\omega\,e_3,
\end{aligned}
\right.
\end{equation}
where $\lambda_3=\sqrt{1-3b^2c}\ge0$. Moreover, $\nabla^\perp\xi_0=0$ if
and only if $\lambda_3=0$.
\end{lemma}

\begin{proof} From the proof of Proposition \ref{prop:4.1},
regardless of $\nabla^\bot\xi_0=0$ or $\nabla^\bot\xi_0\ne0$,
we can always choose a unitary frame $\{e_0,e_A\}$ such that
$\varphi=[e_0]$ and \eqref{eqn:4.14} holds.

Proposition \ref{prop:4.1} shows that $c_{12}^1=c_{13}^1=c_{21}^2=0$
and $c_{31}^2=c_{12}^3$, which implies that $\{\omega_i\}$ is a
normalized frame. Next, following the proof of Corollary
\ref{cor:4.1}, we may assume $c_{23}^1<0$. Then, by Proposition
\ref{prop:2.1}, we see that $c:=c_{23}^1c_{31}^2/4$ and
$b:=-c_{23}^1/(2\sqrt{c})$ are positive real numbers, and so that we
have \eqref{eqn:4.15}.

Now, \eqref{eqn:3.16} and the first equation of \eqref{eqn:3.19}
become
\begin{equation}\label{eqn:4.17}
\omega_{23}=(c_{23}^1+a)\omega_1,\ \ \theta_{12}=-\sigma=ib\sqrt{c}\omega,\ \
d\omega_1=ib\sqrt{c}\omega\wedge\bar{\omega}.
\end{equation}

It follows that, by \eqref{eqn:3.33} and that $d$ is an isomorphism,
$\rho_0=\tfrac1{b\sqrt{c}}\omega_1$. Hence, by \eqref{eqn:4.2},
we have
$\rho_1=\rho_0+c_{23}^1\omega_1=\tfrac{1-2b^2c}{b\sqrt{c}}\omega_1$.

From \eqref{eqn:4.14}, \eqref{eqn:4.17} and the above, we can use
the first equation of \eqref{eqn:3.36} to conclude that
$\lambda_3^2=1-3b^2c$. Finally, the assertion that $\nabla^\perp\xi_0=0$
is equivalent to $\lambda_3=0$ follows from \eqref{eqn:4.16} and
\eqref{eqn:3.32}.
\end{proof}

If $n=2$, then by definition $\nabla^\bot\xi_0=0$ is trivially
satisfied. Moreover, case (ii) in Corollary \ref{cor:4.1} can be
looked as a special case of Lemma \ref{lem:4.2} with $c=4/3$,
$b=\tfrac12$ and $\lambda_3=0$. In summary, we have the following
result.

\begin{theorem}\label{thm:4.1}
Let $\varphi:S^3\rightarrow\mathbb CP^n$ be an equivariant CR minimal immersion
with full dimension $n$. Then $\varphi$ has to be in one of the
following three cases:

{\rm (1)} $n=2$ and $\varphi(S^3)$ is not a Berger sphere case;

{\rm (2)} $n\ge2$ and $\varphi(S^3)$ is a Berger sphere with
$\nabla^\bot\xi_0=0$;

{\rm (3)} $n\ge3$ and $\varphi(S^3)$ is a Berger sphere with
$\nabla^\bot\xi_0\ne0$.
\end{theorem}

\begin{remark}\label{rm:4.1}
In Theorem 4 of \cite{K-R}, Kim and Ryan proved that the only
compact pseudo-Einstein hypersurfaces in $\mathbb{C}P^2$ are the
geodesic spheres. We would point out that these geodesic spheres are
Berger spheres of CR type (see, e.g., pp. 95-96 of \cite{DSA}), and
some of them are minimal depending on their radii.
\end{remark}

\section{Existence of equivariant CR minimal immersions}\label{sect:5}

In this section we show that all the three cases in Theorem
\ref{thm:4.1} do exist. Indeed, for each case the equivariant CR
minimal immersions can be expressed explicitly in the form of
polynomials. As the round sphere is a special Berger sphere, our
results extend that of \cite{Li} we have mentioned in the
introduction.

Let $\{X'_i\}$ be the basis of $\mathfrak{su}(2)$ defined by
\eqref{eqn:2.5}, and set $Z'=X'_2+iX'_3$. Then
\begin{equation}\label{eqn:5.1}
X'_1=i\left(z\tfrac\partial{\partial z}+w\tfrac\partial{\partial w}-\bar
z\tfrac\partial{\partial\bar z}-\bar w\tfrac\partial{\partial\bar w}\right),\ \
Z'=-\bar w\tfrac\partial{\partial z}+\bar z\tfrac\partial{\partial w}.
\end{equation}

Let $\{\varepsilon_0,\varepsilon_1,\ldots,\varepsilon_n\}$ be the natural basis of
$\mathbb C^{n+1}$ and $C_\alpha^n=n!/[\alpha!(n-\alpha)!]$ be the binomial
coefficients. Define a $\mathbb C^{n+1}$-valued function
$$
f=f(z,w):=\sum_{\alpha=0}^n\sqrt{C_\alpha^n}\,z^{n-\alpha}w^\alpha,\ \
(z,w)\in S^3,
$$
and then set
$$
f_0:=f,\ \ f_{-1}=f_{n+1}:=0,\ \
f_{\alpha}:=\tfrac1{\alpha!\sqrt{C_{\alpha}^n}}Z'^{\alpha}f,\ \ \alpha=1,\ldots,n.
$$

Applying Lemma 3.1 of \cite{Li} and the proof of Lemma 3.3 in
\cite{Li}, we have the following lemma, where $\{\omega'_i\}$ is the
dual frame of $\{X'_i\}$ and $\omega'=\omega'_2+i\omega'_3$.

\begin{lemma}\label{lem:5.1}
For $n\ge2$, $\{f_0,f_1,\ldots,f_n\}$ is a unitary frame of
$\underline{\mathbb C}^{n+1}$. Moreover,
\begin{equation}\label{eqn:5.2}
df_{\alpha}=-\sqrt{\alpha(n+1-\alpha)}\bar\omega'f_{\alpha-1}+i(n-2\alpha)\omega'_1f_{\alpha}+\sqrt{(\alpha+1)(n-\alpha)}\omega'f_{\alpha+1},
\end{equation}
for $\alpha=0,1,\ldots,n$.
\end{lemma}

\subsection{Existence of equivariant CR minimal non-Berger
sphere}\label{sect:5.1}~

Define
$$
\begin{aligned}
&f_0=(z^2,\sqrt{2}zw,w^2),\ \ f_1=(-\sqrt 2z\bar w,z\bar z-w\bar
w,\sqrt 2w\bar z),\ \ f_2=(\bar w^2,-\sqrt 2\bar z\bar w,\bar z^2).
\end{aligned}
$$

By Lemma \ref{lem:5.1}, $\{f_0,f_1,f_2\}$ is a unitary frame of
$\underline{\mathbb C}^{3}$. Moreover, it holds that
\begin{equation}\label{eqn:5.3}
df_0=2i\omega'_1f_0+\sqrt2\omega'f_1,\
df_1=-\sqrt2\bar\omega'f_0+\sqrt2\omega'f_2, \
df_2=-\sqrt2\bar\omega'f_1-2i\omega'_1f_2.
\end{equation}

Now, we define ${e_0}:{S^3}\rightarrow{\mathbb C^3}$ by
\begin{equation}\label{eqn:5.4}
e_0=\cos\tfrac\pi8f_0+i\sin\tfrac\pi8f_2
=\cos\tfrac\pi8(z^2,\sqrt{2}zw,w^2)+i\sin\tfrac\pi8(\bar w^2,-\sqrt 2\bar z\bar w,\bar z^2).
\end{equation}

\begin{proposition}\label{prop:5.1}
The $\mathbb C^3$-valued function $e_0$, given by \eqref{eqn:5.4},
defines an equivariant CR minimal immersion
${\varphi=[e_0]}:{S^3}\rightarrow{\mathbb C P^2}$, which is not a Berger sphere.
\end{proposition}

\begin{proof} Let $t=\pi/8$ and that
\begin{equation}\label{eqn:5.5}
e_1=i\sin t\kern1pt f_0+\cos t\kern1pt f_2,\ \ e_2=e^{-it}f_1.
\end{equation}
Then, by virtue of Lemma \ref{lem:5.1}, $\{e_0,e_1,e_2\}$ is a
unitary frame of $\underline{\mathbb C}^{3}$. Solving \eqref{eqn:5.4} and
\eqref{eqn:5.5}, we get
$$
f_0=\cos t\kern2pt e_0-i\sin t\kern2pt e_1,\ \ f_1=e^{it}e_2, \ \
f_2=-i\sin t\kern2pt e_0+\cos t\kern2pt e_1.
$$

Using \eqref{eqn:5.3}\,--\,\eqref{eqn:5.5}, and noting that
$\sqrt2e^{2it}=1+i$, we have
\begin{equation}\label{eqn:5.6}
\left\{
\begin{array}{@{\extracolsep{1pt}}l}
de_0=i\kern2pt\omega_1e_0+ \omega_1e_1+\omega e_2,\\[2pt]
de_1=-\omega_1e_0-i\kern1pt\omega_1e_1-(\omega_2+2\omega_3-i\omega_3)e_2,\\[2pt]
de_2=-\kern1pt\bar\omega\, e_0+(\omega_2+2\omega_3+i\omega_3)e_1,
\end{array}\right.
\end{equation}
where $\omega_1=\sqrt2\omega'_1,\ \omega_2=\sqrt2\omega'_2-\omega'_3,\
\omega_3=\omega'_3,\ \omega=\omega_2+i\omega_3$.

Then, by Proposition 4.2 of \cite{Li} we see that $\varphi$ is an
equivariant immersion, with induced metric $ds^2=\sum\omega_i^2$ and an
orthonormal frame $\{\omega_i\}$.

From \eqref{eqn:5.6} we know that $-\varphi^*\Omega=\tfrac
i2\omega\wedge\bar\omega=\omega_2\wedge\omega_3$. Thus, $\varphi$ is of CR type.

Comparing \eqref{eqn:3.7} with \eqref{eqn:5.6}, and noting that
$\omega_{12}+i\omega_{13}=\sigma=-\theta_{12}$, we have
$$
\rho_0=\omega_1=-\rho_1,\ \ \omega_{12}=\omega_2+2\omega_3,\ \ \omega_{13}=-\omega_3.
$$

From the above equations and \eqref{eqn:4.1} we see that $\varphi$
is minimal.

Finally, as $\omega\wedge\sigma=\sqrt2e^{2it}\omega\wedge\bar\omega\ne0$, by
\eqref{eqn:3.16}, \eqref{eqn:3.17} and Proposition \ref{prop:3.1},
we see that $(S^3,ds^2)$ is not a Berger sphere.

Moreover, from \eqref{eqn:5.6} and Proposition \ref{prop:3.2} (ii),
it is easily seen that $\varphi(S^3)$ is totally geodesic with
respect to $W_2$.
\end{proof}

\subsection{Existence of equivariant CR minimal Berger
sphere}\label{sect:5.2}~

In this subsection, we agree with the following range of indices
$$
\alpha,\beta,\ldots=0,1,\ldots,k;\quad \alpha',\beta',\ldots=0,1,\ldots,\ell,
$$
whereas the indices $A,B,i,j$ having the convention as before.

Let $k,\ell$ be two integers with $k>\ell\geq0$. Suppose
$\{\varepsilon_0,\ldots,\varepsilon_k,\varepsilon'_0,\ldots,\varepsilon'_\ell\}$ is the natural
basis of $\mathbb C^{n+1}=\mathbb C^{k+1}\oplus\mathbb C^{\ell+1}$ with
$n=k+\ell+1$ and $\varepsilon'_{\alpha'}=\varepsilon_{k+1+\alpha'}$.

Define two functions $f:{S^3}\rightarrow{\mathbb C^{n+1}}$ and $
h:{S^3}\rightarrow{\mathbb C^{n+1}}$ by
$$
f(z,w)=\sum\sqrt{C_{\alpha}^k}z^{k-\alpha}w^\alpha\varepsilon_\alpha,\ \
h(z,w)=\sum\sqrt{C_{\alpha'}^\ell}z^{\ell-\alpha'}w^{\alpha'}\varepsilon'_{\alpha'}.
$$

According to Lemma \ref{lem:5.1},
$\{f_0,f_1,\ldots,f_k,h_0,\ldots,h_\ell\}$ is a unitary frame of
$\underline{\mathbb C}^{n+1}$.

For $t\in(0,\pi/2)$, we define a $\mathbb C^{n+1}$-valued function $e_0$
by
\begin{equation}\label{eqn:5.7}
e_0:=e_0(t)=\cos t\kern3pt f_0+i\sin t\kern3pt h_0.
\end{equation}

\begin{proposition}\label{prop:5.2}
For each $t\in(0,\pi/2)$, the function $e_0$, given by
\eqref{eqn:5.7}, defines an equivariant CR immersion
$\varphi=[e_0]:S^3\rightarrow\mathbb CP^n$ with the following properties:

{\rm (i)}\ $(S^3,ds^2)$ is a Berger sphere;

{\rm (ii)}\ $\nabla^{\bot}\xi_0=0$ if and only if $\ell=0$;

{\rm (iii)}\ $\varphi$ is minimal if and only if $t$ satisfies that
\begin{equation}\label{eqn:5.8}
\tan^2 t=2k/\big[3(k-\ell)+\sqrt{(k+\ell)^2+8(k-\ell)^2}\big].
\end{equation}
\end{proposition}

\begin{proof} There is a unique $t_1\in[0,\pi/2)$ such that
\begin{equation}\label{eqn:5.9}
\cos t_1=\tfrac{\sqrt k\cos t}{\sqrt{k\cos^2t+\ell\sin^2t}},\ \ \sin
t_1=\tfrac{\sqrt\ell\sin t}{\sqrt{k\cos^2t+\ell\sin^2t}}.
\end{equation}

Now we put
\begin{equation}\label{eqn:5.10}
\left\{
\begin{aligned}
&e_1=i\sin t\kern1pt f_0+\cos t\kern1pt h_0,\ \ e_2=\cos
t_1f_1+i\sin
t_1h_1,\\
&e_3=i\sin t_1f_1+\cos t_1h_1,\ \ e_4=f_2,\ \ e_5=h_2,\\
&e_6=f_3,\ \ldots,\ e_{k+3}=f_k,\ \ e_{k+4}=h_3,\ \ldots,\
e_n=h_{\ell}.
\end{aligned}
\right.
\end{equation}
Then $\{e_0,e_A\}$ is a unitary frame of $\underline{\mathbb C}^{n+1}$ due to that
$\{f_\alpha,h_{\alpha'}\}$ is unitary.

From similar expressions as \eqref{eqn:5.2} we can see that
$\langle{df_\alpha},{\bar f_\beta}\rangle$, $\langle{dh_{\alpha'}},{\bar h_{\beta'}}\rangle$ are
all left-invariant one-forms, and that $\langle{df_\alpha},{\bar
h_{\alpha'}}\rangle=\langle{dh_{\alpha'}},{\bar f_\alpha}\rangle=0$. It follows that
$\langle{de_0},{\bar e_0}\rangle$, $\langle{de_0},{\bar e_A}\rangle$,
$\langle{de_A},{\bar e_B}\rangle\in \mathfrak{su}(2)^*\otimes\mathbb C$. Hence
$\varphi$ is equivariant.

From \eqref{eqn:5.7} and \eqref{eqn:5.10}, with the use of
\eqref{eqn:5.2} and \eqref{eqn:5.9}, we can easily verify that
\begin{equation}\label{eqn:5.11}
de_0=\cos t\kern2pt df_0+i\sin t\kern2pt
dh_0=ia'_0\omega'_1e_0+m\sin2t\kern2pt\omega'_1e_1+\sqrt{a'_0}\kern2pt\omega'e_2,
\end{equation}
where
\begin{equation}\label{eqn:5.12}
a'_0=k\cos^2t+\ell\sin^2t>0, \ \ m=\tfrac12(k-\ell)>0.
\end{equation}

From \eqref{eqn:5.11}, we see that $\varphi$ is an immersion with the
induce metric
\begin{equation}\label{eqn:5.13}
ds^2=m^2\sin^22t\kern2pt\omega'^2_1+a'_0(\omega'^2_2+\omega'^2_3).
\end{equation}

Obviously, $ds^2$ is a Berger metric with $c=1/a'_0$ and
$b=m\sqrt{c}\,\sin2t$ in \eqref{eqn:2.17}.

Let $\{\omega_i\}$ be an orthonormal frame with respect to $ds^2$,
where $\omega_1=m\sin2t\kern2pt \omega'_1$,
$\omega=\omega_2+i\omega_3=\sqrt{a'_0}\kern2pt\omega'$. Then \eqref{eqn:5.11}
can be rewritten as
\begin{equation}\label{eqn:5.14}
de_0=i\rho_0e_0+\omega_1e_1+\omega\kern2pt e_2,\ \
\rho_0=\tfrac{k\cos^2t+\ell\sin^2t}{m\sin2t}\kern2pt\omega_1.
\end{equation}

It follows that $-\varphi^*\Omega=(i/2)\omega\wedge\bar\omega=\omega_2\wedge\omega_3$.
Thus $\varphi$ is of CR type.

Similarly, using \eqref{eqn:5.2}, \eqref{eqn:5.9} and $de_1=i\sin
t{\kern2pt}df_0+\cos t{\kern2pt}dh_0$, we can verify that
\begin{equation}\label{eqn:5.15}
de_1=-\omega_1e_0+i\rho_1e_1+\tfrac{im\sin2t}{k\cos^2t+\ell\sin^2t}\kern2pt\omega\kern2pt
e_2 +\lambda_3\omega\kern2pt e_3,
\end{equation}
where
\begin{equation}\label{eqn:5.16}
\rho_1=\tfrac{k\sin^2t+\ell\cos^2t}{m\sin2t}\kern2pt\omega_1, \ \
\lambda_3=\tfrac{\sqrt{k\ell}}{k\cos^2t+\ell\sin^2t}.
\end{equation}

Then, Proposition \ref{prop:3.2} shows that $\nabla^{\bot}\xi_0=0$
if and only if $\ell=0$.

Finally, we have the calculation
\begin{equation}\label{eqn:5.17}
\begin{aligned}
de_2=-\bar\omega\kern2pt
e_0&+\tfrac{im\sin2t}{k\sin^2t+\ell\cos^2t}\bar\omega\kern2pt
                  e_1+i(k\sin^2t_1+\ell\cos^2t_1-2)\omega'_1e_2\\
& +\lambda_3\omega_1e_3+\mu_4\omega\kern2pt e_4+i\mu_5\omega\kern2pt e_5,
\end{aligned}
\end{equation}
where $\mu_4=\sqrt{2k(k-1)}/a'_0$,
$\mu_5=\sqrt{2\ell(\ell-1)}/a'_0$\ \ (Note: $\mu_5=0$ if $\ell=0$).

From \eqref{eqn:5.15} and \eqref{eqn:5.17} we see that the second
equation in \eqref{eqn:4.1} is satisfied. By using \eqref{eqn:5.14},
\eqref{eqn:5.15} and \eqref{eqn:5.16}, we get
$$
\rho_1-\rho_0=-2\cot 2t\kern2pt\omega_1,\ \
\omega_{12}+i\omega_{13}=-\theta_{12}=\tfrac{m\sin2t}{k\cos^2t+\ell\sin^2t}(\omega_3-i\omega_2).
$$
Therefore, by \eqref{eqn:4.1}, $\varphi$ is minimal if and only if
\begin{equation}\label{eqn:5.18}
\cot2t=\tfrac{m\sin
2t}{k\cos^2t+\ell\sin^2t}=\tfrac{2m\sin2t}{(n-1)+2m\cos2t}.
\end{equation}
This leads to
\begin{equation}\label{eqn:5.19}
\cos2t=[\sqrt{(n-1)^2+32m^2}-(n-1)]/(8m).
\end{equation}
Then, noting that $k-\ell=2m$ and $k+\ell=n-1$, we can easily get
\eqref{eqn:5.8}.

We have completed the proof of Proposition \ref{prop:5.2}.
\end{proof}

In the special case that $\ell=0$ and $k=n-1$, by \eqref{eqn:5.8} we
have $t=\pi/6$ and \eqref{eqn:5.7} becomes
\begin{equation}\label{eqn:5.20}
e_0=\tfrac{\sqrt3}2\sum\sqrt{C_{\alpha}^{n-1}}z^{n-1-\alpha}w^\alpha\varepsilon_\alpha+\tfrac i2\varepsilon_n.
\end{equation}

From Proposition \ref{prop:5.2} and its proof we have the following
\begin{corollary}\label{cor:5.1}
The function $e_0$, given by \eqref{eqn:5.20}, defines an
equivariant CR minimal immersion $\varphi=[e_0]:S^3\rightarrow\mathbb CP^n$,
which induces a Berger metric $ds^2$. Moreover, it satisfies
$\nabla^{\bot}\xi_0=0$.
\end{corollary}

\begin{remark}\label{rem:5.1}
{\rm (a) The particular case
$e_0=\tfrac{\sqrt3}2(z,w,0)+\tfrac{i}2(0,0,1)$, which corresponds to
$n=2$ and $t=\pi/6$, is exactly the case (ii) of Corollary
\ref{cor:4.1}.

\vskip 1mm

(b) From \eqref{eqn:5.13} one can see that the induced metric has
constant curvature $c$ if and only if
$\tfrac1c=m^2\sin^22t=k\cos^2t+\ell\sin^2t=\tfrac{n-1}2+m\cos2t$.
This, by \eqref{eqn:5.18}, is equivalent to $\cot2t=\sqrt c$. If it
does occur, then $m\cos2t=1$, $n=2m^2-3$ and $c=1/(m^2-1)$ for some
$m\ge2$. We then recover Theorem 6.1 in \cite{Li}.

\vskip 1mm

(c) The function $e_0=e_0(t)$ in \eqref{eqn:5.7} depends on the
parameter $t$. For a fixed pair $k,\ell$ with $k>\ell\ge0$ there is
a family of equivariant CR immersion
$\{\varphi_t=[e_0(t)]\kern1pt|\kern1pt t\in(0,\pi/2)\}$. Each of them
induces a Berger metric. However, among them only one is minimal by
\eqref{eqn:5.8}. Moreover, for fixed $n\ge2$ there are $[n/2]$ such
families, given by $\ell=0,\, \ldots,\,[n/2]-1$. Here, $[n/2]$
denotes the largest integer not exceeding $n/2$.}
\end{remark}

\section{Uniqueness of equivariant CR minimal immersions}\label{sect:6}

Recall that two immersions $\psi_1,\psi_2:S^3\rightarrow\mathbb CP^n$ are said
to be equivalent, denoted by $\psi_1\sim\psi_2$, if there exists a
holomorphic isometry $A:\mathbb CP^n\rightarrow\mathbb CP^n$ such that $\psi_1=A\circ\psi_2$.

For simplicity, we denote by $\varphi_1$, $\varphi_2$, $\varphi_3$ the three
equivariant CR minimal immersions determined by \eqref{eqn:5.4},
\eqref{eqn:5.20}, \eqref{eqn:5.7} with the additional conditions
\eqref{eqn:5.8} and $\ell>0$, respectively. Let $X'_1,Z'$ be the
differential operators as defined by \eqref{eqn:5.1}.

First of all, we introduce Lemma 3.3 in \cite{Li} as below.

\begin{lemma}[\cite{Li}]\label{lem:6.1}
Let $f$ be a $\mathbb C^{n+1}$-valued function defined on $S^3$. If it
satisfies $\langle f,{\bar f}\rangle=1,\bar Z'f=0,\,X'_1f=isf$, then $s$ is a
nonnegative integer with $s\le n$. Moreover, there is a unitary
basis $\{\varepsilon_0,\ldots,\varepsilon_s,\varepsilon_{s+1},\ldots,\varepsilon_n\}$ of
$\mathbb C^{n+1}$ such that
$$
f(z,w)=\sum_{\alpha=0}^s\sqrt{C_{\alpha}^s}z^{s-\alpha}w^\alpha\varepsilon_\alpha,\ \
(z,w)\in S^3.
$$
\end{lemma}

Next, we will prove two uniqueness theorems based on preceding
Theorem \ref{thm:4.1}. For that purpose, in sequel we assume that
$\varphi:S^3\rightarrow\mathbb CP^n$ is an equivariant CR minimal immersion with
full dimension $n$.

\begin{theorem}\label{thm:6.1}
If $\varphi(S^3)$ is not a Berger sphere, then up to an inner
automorphism of $SU(2)$ we have $\varphi\sim\varphi_1=[e_0]$, where $e_0$ is
defined by \eqref{eqn:5.4}.
\end{theorem}

\begin{proof} According to Proposition \ref{prop:4.2}, we have $n=2$.
Moreover, by Corollary \ref{cor:4.1} (i), we have a unitary frame
$\{e_0,e_1,e_2\}$ of $\underline{\mathbb C}^3$ such that $\varphi=[e_0]$ and
\begin{equation}\label{eqn:6.1}
\left\{
\begin{aligned}
de_0&=i \omega_1  e_0 +\omega_1e_1 +\omega e_2,\\
de_1&=- \omega_1  e_0- i\omega_1e_1+(i\omega+\sqrt2\bar\omega) e_2,\\
de_2&=-\bar\omega e_0 -(\sqrt2\omega-i\bar\omega)e_1.
\end{aligned}
\right.
\end{equation}

Now, we take another unitary frame $\{\tilde e_0,\tilde e_1,\tilde e_2\}$ of
$\underline{\mathbb C}^3$ by
\begin{equation}\label{eqn:6.2}
\tilde e_0=\cos\tfrac\pi8\;e_0-i\sin\tfrac\pi8\;e_1,\ \ \tilde
e_1=-i\sin\tfrac\pi8\;e_0+\cos\tfrac\pi8\;e_1,\ \ \tilde e_2=e_2,
\end{equation}
or, equivalently,
\begin{equation}\label{eqn:6.3}
e_0=\cos\tfrac\pi8\;\tilde e_0+i\sin\tfrac\pi8\;\tilde e_1,\ \
e_1=i\sin\tfrac\pi8\;\tilde e_0+\cos\tfrac\pi8\;\tilde e_1,\ \ e_2=\tilde
e_2.
\end{equation}

By setting $\omega'_1=\frac1{\sqrt{2}}\omega_1$,\
$\omega'=\omega'_2+i\omega'_3=\cos\frac\pi8\omega-i\sin\frac\pi8\bar\omega$, or
equivalently,
$$
\omega'_1=\tfrac1{\sqrt{2}}\omega_1,\ \
\omega'_2=\cos\tfrac\pi8\,\omega_2-\sin\tfrac\pi8\,\omega_3,\ \
\omega'_3=-\sin\tfrac\pi8\,\omega_2+\cos\tfrac\pi8\,\omega_3,
$$
then, using that $\cos\frac\pi8+\sin\frac\pi8=\sqrt2\cos\frac\pi8$
and $\cos\frac\pi8-\sin\frac\pi8=\sqrt2\sin\frac\pi8$, we obtain
$2\omega'_1\wedge\omega'_2\wedge\omega'_3=\omega_1\wedge\omega_2\wedge\omega_3$. This
shows that $\{\omega'_1,\omega'_2,\omega'_3\}$ is a basis of
$\mathfrak{su}(2)^*$.

Next, direct calculations show that
\begin{equation}\label{eqn:6.4}
\left\{
\begin{aligned}
d\tilde e_0&=2i\omega_1'\tilde e_0 + \sqrt2\omega' \tilde e_2,\\
d\tilde e_1&=-\ 2i\omega'_1 \tilde e_1\,- \sqrt2\bar\omega'\tilde e_2,\\
d\tilde e_2&=-\sqrt2\bar\omega'\tilde e_0 +\sqrt2\omega' \tilde e_1.
\end{aligned}
\right.
\end{equation}

On the other hand, exterior differentiations of \eqref{eqn:6.3} give
that $d\omega'_1=i\omega'\wedge\bar\omega'$, $d\omega'=2i\omega'_1\wedge\omega'$.
Hence $\{\omega'_i\}$ satisfies \eqref{eqn:2.3}. By Lemmas
\ref{lem:2.1} and \ref{lem:2.2}, up to an inner automorphism of
$SU(2)$ we may assume that the dual frame $\{X'_i\}$ of $\{\omega'_i\}$
is defined by \eqref{eqn:2.5}, or equivalently by \eqref{eqn:5.1}.
Then we have
\begin{equation}\label{eqn:6.5}
\bar Z'\tilde e_0=0,\ \  X'_1\tilde e_0=2i\tilde e_0,\ \ Z'\tilde e_0=\sqrt2\tilde
e_2,\ \ Z'\tilde e_2=\sqrt2\tilde e_1.
\end{equation}

Using Lemma \ref{lem:6.1} we then have
$$
\tilde e_0=z^2\varepsilon_0+\sqrt2zw\varepsilon_1+w^2\varepsilon_2,\ \ (z,w)\in S^3,
$$
where $\{\varepsilon_0,\varepsilon_1,\varepsilon_2\}$ is a unitary basis of $\mathbb C^3$. By
re-choosing a basis of $\mathbb C^3$, or equivalently by a holomorphic
isometry $A:{\mathbb C^3}\rightarrow{\mathbb C^3}$, we may set
$$
\tilde e_0=\big(z^2,\sqrt2zw,w^2\big).
$$
Then, by \eqref{eqn:6.5}, we have
$$
\tilde e_2=\tfrac1{\sqrt2}Z'\tilde e_0=(-\sqrt2z\bar w,z\bar z-w\bar
w,\sqrt2w\bar z),\ \tilde e_1=\tfrac1{\sqrt2}Z'\tilde e_2=(\bar
w^2,-\sqrt2\bar z\bar w,\bar z^2).
$$
It follows from \eqref{eqn:6.3} that
$$
e_0=\cos\tfrac\pi8\,(z^2,\sqrt{2}zw,w^2)+i\sin\tfrac\pi8\,(\bar{w}^2,-\sqrt{2}\bar{z}\bar{w},\bar{z}^2).
$$

This completes the proof of Theorem \ref{thm:6.1}.
\end{proof}

\begin{theorem}\label{thm:6.2}
If $\varphi(S^3)$ is a Berger sphere with $\nabla^\bot\xi_0=0$ {\rm
(}resp. with $\nabla^\bot\xi_0\ne0${\rm )}, then, up to an inner
automorphism of $SU(2)$, we have $\varphi\sim\varphi_2=[e_0]$ and $e_0$ is
defined by \eqref{eqn:5.20} {\rm(}\,resp. $\varphi\sim\varphi_3=[e_0]$ and
$e_0$ is defined by \eqref{eqn:5.7} and \eqref{eqn:5.8} with some
$\ell>0${\rm)}.
\end{theorem}

\begin{proof} Since $\varphi(S^3)$ is a Berger sphere, according to Corollary \ref{cor:4.1} (ii) and Lemma \ref{lem:4.2},
regardless of $n=2$ or $n\ge3$, we have a normalized frame
$\{\omega_i\}$ and a unitary frame $\{e_0,e_A\}$ such that $\varphi=[e_0]$
and
\begin{equation}\label{eqn:6.6}
d\omega_1=2b\sqrt c\omega_2\wedge\omega_3,\ \ d\omega_2=\tfrac{2\sqrt
c}b\omega_3\wedge\omega_1,\ \ d\omega_3=\tfrac{2\sqrt c}b\omega_1\wedge\omega_2;
\end{equation}
\begin{equation}\label{eqn:6.7}
\left\{
\begin{aligned}
de_0=&\tfrac i{b\sqrt c}\omega_1e_0+\omega_1e_1+\omega\,e_2, \\
de_1=&-\omega_1e_0+\tfrac{i(1-2b^2c)}{b\sqrt
c}\omega_1e_1+ib\sqrt{c}\omega\,e_2\,+\lambda_3\omega\,e_3,
\end{aligned}
\right.
\end{equation}
where $b,c$ are positive real numbers with $b=\tfrac12, c=\tfrac43$
if $n=2$; $\lambda_3=\sqrt{1-3b^2c}\ge0$ and that $\lambda_3=0$ if and only
if $\nabla^\bot\xi_0=0$.

If we set
\begin{equation}\label{eqn:6.8}
\omega'_1=\tfrac{\sqrt c}b\kern2pt\omega_1,\ \ \omega'_2=\sqrt
c\kern2pt\omega_2,\ \ \omega'_3=\sqrt c\kern2pt\omega_3,
\end{equation}
then $\{\omega'_i\}$ is a basis of $\mathfrak{su}(2)^*$. Using
\eqref{eqn:6.6} we see that $\{\omega'_i\}$ satisfy \eqref{eqn:2.3}. As
before, up to an inner automorphism of $SU(2)$, we may assume that
the dual frame $\{X'_i\}$ of $\{\omega'_i\}$ is defined by
\eqref{eqn:5.1}.

From \eqref{eqn:6.7} and \eqref{eqn:6.8} we get
\begin{equation}\label{eqn:6.9}
\left\{
\begin{aligned}
de_0=& \tfrac ic \omega'_1e_0+\tfrac b{\sqrt c}\omega'_1e_1+\tfrac 1{\sqrt c}\omega'e_2, \\
de_1=&-\tfrac b{\sqrt c}\omega'_1e_0+\tfrac{i(1-2b^2c)}c\omega'_1e_1+ib
\omega'e_2\,+\tfrac{\lambda_3}{\sqrt c}\,\omega'e_3,
\end{aligned}
\right.
\end{equation}
where $\omega'=\omega'_2+i\omega'_3$. Hence we have
\begin{equation}\label{eqn:6.10}
\bar Z'e_0=0,\ \ \bar Z'e_1=0,
\end{equation}
\begin{equation}\label{eqn:6.11}
X'_1e_0=ic^{-1}e_0+ bc^{-1/2}e_1, \ \ X'_1e_1=bc^{-1/2}
e_0+i(c^{-1}-2b^2)e_1.
\end{equation}

Let $V$ be the subbundle of $\underline{\mathbb C}^{n+1}$ spanned by $\{e_0,e_1\}$.
Then \eqref{eqn:6.11} defines an endomorphism ${X_1}:V\rightarrow V$. It is
skew-Hermitian, and therefore has two imaginary eigenvalues
$ik,i\ell$ such that
\begin{equation}\label{eqn:6.12}
k=c^{-1}-b^2+b\sqrt{c^{-1}+b^2},\ \
\ell=c^{-1}-b^2-b\sqrt{c^{-1}+b^2}.
\end{equation}

Since $k>c^{-1}>0$ and $k\ell=c^{-2}(1-3b^2c)\ge0$, we see that
$k>\ell\ge0$, and $\ell=0$ if and only if $\lambda_3=0$, or
equivalently, $\nabla^\bot\xi_0=0$.

It is well known that there is a unitary frame $\{f,h\}$ of $V$ such that
\begin{equation}\label{eqn:6.13}
X'_1f=ik\kern2pt f, \ \ X'_1h=i\ell\kern2pt h.
\end{equation}

Since by theory of linear algebra $(X'_1-ik\kern2pt
I)(X'_1-i\ell\kern2pt I)=0$, and, by \eqref{eqn:6.11} and
\eqref{eqn:6.12}, we have
$$
(X'_1-i\ell\kern2pt I)e_0=ib[(b+\sqrt{c^{-1}+b^2})e_0-ic^{-1/2}e_1],
$$
$$
(X'_1-ik\kern2pt I)e_0=b[-i(\sqrt{c^{-1}+b^2}-b)e_0+c^{-1/2}e_1].
$$

We may choose $\{f,h\}$ such that
\begin{equation}\label{eqn:6.14}
f=\cos t\kern2pt e_0-i\sin t\kern2pt e_1,\ \ h=-i\sin t\kern2pt
e_0+\cos t\kern2pt e_1,
\end{equation}
where $t\in(0,\pi/2)$ is determined by
\begin{equation}\label{eqn:6.15}
\tan t=c^{-1/2}/(b+\sqrt{c^{-1}+b^2})=\sqrt{1+b^2c}-b\sqrt c.
\end{equation}

Then, $|f|=|h|=1$, $\langle f,{\bar h}\rangle=0$. Moreover, by
\eqref{eqn:6.10} and \eqref{eqn:6.13}, we get
$$
\bar Z'f=\bar Z'h=0, \ \ X'_1f=ik\kern1pt f, \ \ X'_1h=i\ell\kern1pt
h.
$$
Since $n$ is the full dimension, applying Lemma \ref{lem:6.1} we see
that $k,\ell$ are integers with $k>\ell\ge0$ and $k+\ell=n-1$.
Moreover, there is a unitary basis
$\{\varepsilon_0,\ldots,\varepsilon_k,\varepsilon'_0,\ldots,\varepsilon'_\ell\}$ of $\mathbb C^{n+1}$
such that
$$
f=\sum_{\alpha=0}^k\sqrt{C^k_\alpha}\,z^{k-\alpha}w^\alpha\varepsilon_\alpha,\ \
h=\sum_{\alpha'=0}^\ell\sqrt{C^\ell_{\alpha'}}\,z^{\ell-\alpha'}w^{\alpha'}\varepsilon'_{\alpha'}.
$$
Thus, by \eqref{eqn:6.14}, we can solve $e_0$ to obtain that
\begin{equation}\label{eqn:6.16}
e_0=\cos t\sum_{\alpha=0}^k\sqrt{C^k_\alpha}z^{k-\alpha}w^\alpha\varepsilon_\alpha +i\sin
t\sum_{\alpha'=0}^\ell\sqrt{C^\ell_{\alpha'}}z^{\ell-\alpha'}w^{\alpha'}\varepsilon'_{\alpha'}.
\end{equation}

If $\nabla^\bot\xi_0=0$, then $\lambda_3=0$ and thus $b=1/\sqrt{3c}$.
Hence $\ell=0$, $t=\pi/6$. Then $\varphi\sim\varphi_2$, due to that
\eqref{eqn:6.16} differs from \eqref{eqn:5.20} by a holomorphic
isometry of $\mathbb C^{n+1}$.

If $\nabla^\bot\xi_0\ne0$, then, by \eqref{eqn:6.12}, $k>\ell>0$.
Since $\varphi$ is minimal, according to $(c)$ of Remark 5.1, the
parameter $t$ determined by \eqref{eqn:6.15} has to be that one
determined by \eqref{eqn:5.8}. Then, using \eqref{eqn:6.12} and
\eqref{eqn:6.15}, one gets
$$
(2k)/[3(k-\ell)+\sqrt{(k+\ell)^2+8(k-\ell)^2}]=1+2b^2c-2bc\sqrt{c^{-1}+b^2}=\tan^2t.
$$

Now, by comparing \eqref{eqn:6.16} with \eqref{eqn:5.7}, we come to
the conclusion $\varphi\sim\varphi_3$.
\end{proof}

\vskip 1mm \noindent{\bf Completion of proof of Theorem
\ref{thm:1.1}}.

Let $\varphi: {S^3}\rightarrow{\mathbb C}P^{n}\ (n\ge2)$ be a
linearly full equivariant CR minimal immersion with induced metric
$ds^2$. If  $(S^3,ds^2)$ is not a Berger sphere, then according to
Theorem \ref{thm:6.1}, we get the assertion (1). If on the other
hand  $(S^3,ds^2)$ is a Berger sphere, then Theorem \ref{thm:6.2}
shows that for both cases of $\nabla^\bot\xi_0=0$ and
$\nabla^\bot\xi_0\ne0$, the expression of $\varphi$ assume the form
as stated in the assertion (2), where $\nabla^\bot\xi_0=0$
corresponds to $\ell=0$, and $\nabla^\bot\xi_0\ne0$ corresponds to
an integer $\ell>0$. \qed



\begin{thebibliography}{99}

\bibitem{Be}
Bejancu, A.: {\it Gemetry of CR-submaniflds}, D.~Reidel Publishing
Company, Dordrecht, 1986.

\bibitem{B-J-R-W}
Bolton, J., Jensen, G. R., Rigoli, M., Woodward, L. M.: {\it On
conformal minimal immersions of $\mathbb{S}^2$ into $\mathbb CP^n$},
Math. Ann. {\bf 279}, 599-620 (1988).

\bibitem{C-D-V-V}
Chen, B.-Y., Dillen, F., Verstraelen, L., Vrancken, L.: {\it An
exotic totally real minimal immersions of $S^3$ in $\mathbb CP^3$
and its characterization}, Proc. Roy. Soc. Edinburgh, Sect. A, {\bf
126}, 153-165 (1996).

\bibitem{C-D-V-V-1}
Chen, B.-Y., Dillen, F., Verstraelen, L., Vrancken, L.: {\it
Lagrangian isometric immersions of a real space form $M^n(c)$ into
complex space form $\widetilde{M}^n(c)$}, Math. Proc. Cambridge Philos.
Soc. {\bf 124}, 107-125 (1998).

\bibitem{DLVW}
Dillen, F., Li, H., Vrancken, L., Wang, X.: {\it Lagrangian
submanifolds in complex projective space with parallel second
fundamental form}, Pacific J. Math. \textbf{255}, 79-115 (2012).

\bibitem{DSA}
Dragomir S., Shahid M.H., F.R. Al-Solamy: {\it Geometry of
Cauchy-Riemann submanifolds}, Springer, Singapore, 2016.

\bibitem{F-P}
Fei, J., Peng, C.-K., Xu, X.-W.: {\it Equivariant totally real
$3$-spheres in the complex projective space $\mathbb{C}P^n$}, Diff.
Geom. Appl. {\bf 30}, 262-273 (2012).

\bibitem{H-Li}
Hu, S., Li, K.: {\it The minimal $S^3$ with constant sectional
curvature in $\mathbb{C}P^n$}, J. Aust. Math. Soc. {\bf 99}, 63-75
(2015).

\bibitem{HLW}
Hu, Z., Lyu, D.-L. Wang, J.: {\it On rigidity phenomena of compact
surfaces in homogeneous $3$-manifolds}, Proc. Amer. Math. Soc. {\bf
143}, 3097-3109 (2015).

\bibitem{J-L}
Jenson, G.-R., Liao, R.: {\it Families of flat minimal tori in
$\mathbb{C}P^n$}, J. Differ. Geom. {\bf 42}, 113-132 (1995).

\bibitem{K-R}
Kim, H.-S., Ryan, P.-J.: {\it A classification of pseudo-Einstein
hypersurfaces in $\mathbb{C}P^2$}, Differential Geom. Appl. {\bf 26}
(2008), 106-112.

\bibitem{Li}
Li, Z.-Q.: {\it Minimal $S^3$ with constant curvature in $\mathbb
CP^n$}, J. London Math. Soc. {\bf 68}, 223-240 (2003).

\bibitem{L-H}
Li, Z.-Q., Huang, A.-M.: {\it Constant curved minimal CR $3$-spheres
in $\mathbb CP^n$}, J. Aust. Math. Soc. {\bf 79}, 1-10 (2005).

\bibitem{L-P}
Li, Z.-Q., Peng, J.-W.: {\it Rigidity of $3$-dimensional minimal
CR-submanifolds with constant curvature in $\mathbb{C}P^n$}, Far
East J. Math. Sci. {\bf 34}, 303-315 (2009).

\bibitem{L-T}
Li, Z.-Q. Tao, Y.-Q.: {\it Equivariant Lagrangian minimal $S^3$
in $\mathbb CP^3$}, Acta Math. Sinica (Engl. Ser.) {\bf 22},
1215-1220 (2006).

\bibitem{L-V-W}
Li, H., Vrancken, L., Wang, X.: {\it A new characterizaiton of the
Berger sphere in complex projective space}, J. Geom. Phys., {\bf
92}, 129-139 (2015).

\bibitem{Ma}
Mashimo, K.: {\it Minimal immersions of $3$-dimensional sphere into
spheres}, Osaka J. Math. {\bf 21}, 721-732 (1984).

\bibitem{Mil}
Milnor, J.: {\it Curvatures of left invariant metrics on Lie
groups}, Adv. Math., {\bf 21}, 293-329 (1976).

\bibitem{Na}
Naitoh, H.: {\it Isotropic submanifolds with parallel second
fundamental form in $P^m(c)$}, Osaka J. Math. {\bf 18}, 427-464
(1981).

\bibitem{N1}
Naitoh, H.: {\it Totally real parallel submanifolds in $P^n(c)$},
Tokyo J. Math. \textbf{4}, 279-306 (1981).

\bibitem{N2}
Naitoh, H.: {\it Parallel submanifolds of complex space forms}, I,
Nagoya Math. J. \textbf{90} (1983), 85-117; II, ibid, \textbf{91}
(1983), 119-149.

\bibitem{O}
Ogiue, K.: {\it Differential geometry of K\"aehler submanifolds},
Adv. Math., {\bf 13}, 73-114 (1974).

\bibitem{T} Torralbo, F.: {\it Compact minimal surfaces in the
Berger spheres}, Ann. Global Anal. Geom. \textbf{41}, 391-405
(2012).

\end{thebibliography}
\end{document}